\documentclass[10pt,reqno]{amsart}
\usepackage[foot]{amsaddr}

\usepackage[utf8]{inputenc}
\usepackage[english]{babel}

\usepackage[left=2cm,right=2cm,top=2cm,bottom=2cm]{geometry}
\usepackage{enumitem}
\setlist[itemize]{leftmargin=2em}

\usepackage{graphics}
\usepackage{float}

\usepackage{color}

\usepackage{url}

\numberwithin{equation}{section}

\usepackage{amsmath}	
\usepackage{amssymb}	
\usepackage{mathtools}	

\usepackage[capitalise,noabbrev]{cleveref}
\crefname{equation}{Eq.\!\!}{Eq.\!\!}
\crefname{assumption}{A.\!\!}{A.\!\!}
\crefname{definition}{D.\!\!}{D.\!\!}
\crefname{lemma}{L.\!\!}{L.\!\!}
\crefname{proposition}{P.\!\!}{P.\!\!}
\crefname{theorem}{T.\!\!}{T.\!\!}
\crefname{corollary}{C.\!\!}{C.\!\!}
\crefname{example}{E.\!\!}{E.\!\!}
\crefname{remark}{R.\!\!}{R.\!\!}

\newtheorem{definition}{Definition}[section]
\newtheorem{assumption}[definition]{Assumption}
\newtheorem{lemma}[definition]{Lemma}

\newtheorem{theorem}[definition]{Theorem}
\newtheorem{corollary}[definition]{Corollary}
\newtheorem{example}[definition]{Example}
\newtheorem{remark}[definition]{Remark}

\usepackage{ifthen}
\newcommand{\ifNotEmpty}[2]{\ifthenelse{\equal{#1}{}}{}{#2}}

\newcounter{repeatCharCounter}
\newcommand{\repeatChar}[2]{\setcounter{repeatCharCounter}{0}\whiledo{\value{repeatCharCounter}<#1}{#2\stepcounter{repeatCharCounter}}}

\renewcommand{\phi}{\varphi}
\renewcommand{\ell}{l}
\renewcommand{\epsilon}{\varepsilon}
\renewcommand{\*}{\hspace{-0.15em}\cdot\hspace{-0.15em}}

\newcommand{\cleq}{\lesssim}
\newcommand{\cgeq}{\gtrsim}
\newcommand{\ceq}{\eqsim}

\renewcommand{\(}{\bigg(}
\renewcommand{\)}{\bigg)}

\newcommand{\set}[1]{\{#1\}}

\newcommand{\Set}[2]{\{#1\,|\,#2\}}

\newcommand{\Pow}[1]{\mathrm{Pow}(#1)}

\newcommand{\N}{\mathbb{N}}

\newcommand{\R}{\mathbb{R}}

\newcommand{\Ball}[3][]{\mathrm{Ball}_{#1}(#2,#3)}

\newcommand{\closureN}[1]{\overline{#1}}
\newcommand{\convexhullN}[1]{\mathrm{conv} #1}
\newcommand{\convexhull}[1]{\mathrm{conv}(#1)}
\newcommand{\boundaryN}[1]{\oldpartial #1}

\newcommand{\inflateN}[3][]{{#2}^{#3}_{#1}}
\newcommand{\inflate}[3][]{(#2)^{#3}_{#1}}

\newcommand{\cardN}[1]{\# #1}

\newcommand{\meas}[2][]{|#2|_{#1}}

\newcommand{\diam}[2][]{\mathrm{diam}_{#1}(#2)}

\newcommand{\dist}[3][]{\mathrm{dist}_{#1}(#2,#3)}

\newcommand{\supp}[2][]{\mathrm{supp}_{#1}(#2)}

\newcommand{\restrictN}[2]{#1|_{#2}}
\newcommand{\restrict}[2]{(#1)|_{#2}}

\newcommand{\fDef}[3]{#1: #2 \longrightarrow #3}
\newcommand{\fDefB}[5][]{
\ifNotEmpty{#1}{#1:}
\left\{
\begin{array}{ccc}
#2 & \longrightarrow & #3 \\
#4 & \longmapsto & #5
\end{array}
\right.
}

\newcommand{\identity}{\mathrm{id}}

\newcommand{\kronecker}[1]{\delta_{#1}}

\newcommand{\Landau}[1]{\mathcal{O}(#1)}

\let\oldprime\prime

\renewcommand{\prime}[2]{(#2)^{\repeatChar{#1}{\oldprime}}}

\let\oldpartial\partial

\renewcommand{\partial}[3][]{\oldpartial_{#2}^{\ifNotEmpty{#1}{(#1)}}(#3)}

\newcommand{\gradN}[2][]{\nabla_{#1} #2}
\newcommand{\grad}[2][]{\nabla_{#1}(#2)}

\renewcommand{\div}[2][]{\mathrm{div}_{#1}(#2)}

\newcommand{\I}[4][]{\int\displaylimits_{#2}^{#1} #3 \,\mathrm{d}#4}

\newcommand{\abs}[1]{|#1|}

\newcommand{\seminorm}[2][]{|#2|_{#1}}

\newcommand{\norm}[2][]{\|#2\|_{#1}}
\newcommand{\normB}[2][]{\bigg\|#2\bigg\|_{#1}}
\newcommand{\skalar}[3][]{\langle#2,#3\rangle_{#1}}

\newcommand{\bilinear}[3][]{#1(#2,#3)}

\newcommand{\ran}[1]{\mathrm{ran}(#1)}
\newcommand{\rank}[1]{\mathrm{rank}(#1)}

\newcommand{\mvemph}[1]{\boldsymbol{#1}}

\renewcommand{\det}[1]{\mathrm{det}(#1)}

\newcommand{\spanN}[1]{\mathrm{span}\,#1}
\renewcommand{\dim}[1]{\mathrm{dim}(#1)}
\newcommand{\dimN}[1]{\mathrm{dim}\,#1}

\newcommand{\Ck}[2]{C^{#1}(#2)}

\newcommand{\CkPw}[2]{C^{#1}_{\mathrm{pw}}(#2)}
\newcommand{\lp}[2]{\ell^{#1}(#2)}
\newcommand{\Lp}[2]{L^{#1}(#2)}
\newcommand{\Hk}[2]{H^{#1}(#2)}
\newcommand{\HkO}[3][]{H^{#2}_{0}(#3 \ifNotEmpty{#1}{,} #1)}

\newcommand{\Wkp}[3]{W^{#1,#2}(#3)}

\newcommand{\Pp}[2]{\mathbb{P}^{#1}(#2)}
\newcommand{\PpO}[2]{\mathbb{P}^{#1}_0(#2)}

\newcommand{\Skp}[3]{\mathbb{S}^{#2,#1}(#3)}
\newcommand{\SkpO}[4][]{\mathbb{S}^{#3,#2}_{0}(#4 \ifNotEmpty{#1}{,} #1)}
\newcommand{\SHarm}[1]{\mathbb{S}_{\mathrm{harm}}(#1)}

\newcommand{\Elements}{\mathcal{T}}

\newcommand{\ElementsB}{\mathcal{B}}

\newcommand{\ElementsS}{\mathcal{S}}
\newcommand{\Nodes}{\mathcal{N}}
\newcommand{\Patch}[2]{#1(#2)}
\newcommand{\h}[1]{h_{#1}}
\newcommand{\hMin}[1]{h_{\min\ifNotEmpty{#1}{,}#1}}
\newcommand{\hMax}[1]{h_{\max\ifNotEmpty{#1}{,}#1}}

\newcommand{\Tree}[2][]{\mathbb{T}_{#2}^{\ifNotEmpty{#1}{(#1)}}}
\newcommand{\depth}[1]{\mathrm{depth}(#1)}
\newcommand{\BPart}{\mathbb{P}}
\newcommand{\BPartAdm}{\mathbb{P}_{\mathrm{adm}}}
\newcommand{\BPartSmall}{\mathbb{P}_{\mathrm{small}}}
\newcommand{\HMatrices}[2]{\mathcal{H}(#1,#2)}

\newcommand{\CSparse}{\sigma_{\mathrm{sparse}}} 
\newcommand{\CPcre}{\sigma_{\mathrm{Pcr}}}      
\newcommand{\CShape}{\sigma_{\mathrm{shp}}}     
\newcommand{\CShapeTilde}{\widetilde{\sigma}_{\mathrm{shp}}}     
\newcommand{\CCard}{\sigma_{\mathrm{card}}}     
\newcommand{\CAdm}{\sigma_{\mathrm{adm}}}       
\newcommand{\CSmall}{\sigma_{\mathrm{small}}}   
\newcommand{\CExp}{\sigma_{\mathrm{exp}}}       
\newcommand{\CRed}{\sigma_{\mathrm{red}}}
\newcommand{\CStab}{\sigma_{\mathrm{stab}}}

\newcommand{\SolutionOp}[1]{S_{#1}}
\newcommand{\CutoffFc}[2]{\kappa_{#1}^{#2}}
\newcommand{\CutoffOp}[2]{K_{#1}^{#2}}
\newcommand{\CoarseningOp}[2]{Q_{#1}^{#2}}

\begin{document}

\title{Exponential meshes and $\mathcal{H}$-matrices}
\author[N. Angleitner, M. Faustmann, J.M. Melenk]{Niklas Angleitner, Markus Faustmann, Jens Markus Melenk}
\address{Institute for Analysis and Scientific Computing (Inst. E 101), Vienna University of Technology, Wiedner Hauptstrasse 8-10, 1040 Wien, Austria}
\email{niklas.angleitner@tuwien.ac.at, markus.faustmann@tuwien.ac.at, melenk@tuwien.ac.at}
\date{\today}

\begin{abstract}
In \cite{Angleitner_H_matrices_FEM}, we proved that the inverse of the stiffness matrix of an $h$-version finite element method (FEM) applied to scalar second order elliptic boundary value problems can be approximated at an exponential rate in the block rank by $\mathcal{H}$-matrices. Here, we improve on this result in multiple ways: (1) The class of meshes is significantly enlarged and includes certain exponentially graded meshes. (2) The dependence on the polynomial degree $p$ of the discrete ansatz space is made explicit in our analysis. (3) The bound for the approximation error is sharpened, and (4) the proof is simplified.
\end{abstract}

\maketitle

\section{Introduction}

${\mathcal H}$-matrices were introduced by W.~Hackbusch in \cite{hackbusch99} as a data-sparse matrix format 
of blockwise low-rank matrices. A particular feature of the ${\mathcal H}$-matrix format is that 
that it comes with an arithmetic that includes the approximate addition, multiplication, and inversion in logarithmic-linear
complexity; we refer to \cite{Hackbusch_Hierarchical_matrices, boermbook, GH03, grasedyck01, bebendorfbook} for a detailed 
discussion of the algorithmic aspects of ${\mathcal H}$-matrices. 
{}A large class of matrices can be represented or at least approximated well in the ${\mathcal H}$-matrix format. 
Discretizations of differential equations can typically be presented exactly, and the matrices from the discretization
of integral operators with so-called asymptotically smooth kernels, which forms a large class of practically relevant
integral operators, can be approximated with an error that is exponentially small in the block rank. Given that 
${\mathcal H}$-matrices come with an approximate arithmetic, it is important to understand, for which matrices that can be 
approximated well in that format, also their inverses can be approximated well. It is the purpose of the present paper
to study this question for matrices arising from Galerkin discretizations of second order elliptic equations on 
strongly graded meshes. 

The question of ${\mathcal H}$-matrix approximability of the inverses of matrices arising in the finite element method (FEM) has attracted 
some attention  in the past. The first results \cite{BebendorfHackbusch03} for scalar elliptic problem 
and \cite{bebendorf2009parallel} for the time-harmonic Maxwell system showed the existence of locally separable approximations 
of the Green's function and inferred from that the approximability of the inverses of the FEM matrices by 
${\mathcal H}$-matrices via a final projection step. This approach generalizes to certain classes of pseudodifferential
operators, \cite{doelz-harbrecht-schwab17}, and results in exponential convergence in the block rank up the 
final projection error. A fully discrete approach, which avoids the final projection steps and leads to exponential convergence in the block rank, was taken
in \cite{Faustmann_H_matrices_FEM,faustmann2021mathcalhmatrix} in a FEM setting on quasi-uniform meshes and in the boundary element method 
(BEM) in \cite{FMP16,FMP17,FMP21}. The generalization of \cite{Faustmann_H_matrices_FEM} to non-uniform meshes 
was achieved in \cite{Angleitner_H_matrices_FEM} for low order FEM on certain classes of meshes that includes 
algebraically graded meshes. In the present work, we generalize \cite{Angleitner_H_matrices_FEM} in several directions: 
first, we admit a larger class of meshes that includes certain shape-regular meshes that are graded exponentially towards
a lower-dimensional manifold. In particular, we can show exponential approximability in the block rank for 
the inverses of FEM matrices arising in variants of the boundary concentrated FEM, \cite{Melenk_Boundary_concentrated_FEM}. 
Second, our analysis is explicit in the polynomial degree $p$. 
For our $p$-explicit analysis, we develop polynomial-preserving lifting and polynomial projection operators on 
simplices in arbitrary spatial dimension. Such operators, generalizing the projection-based operators 
of \cite{demkowicz-babuska03,demkowicz-buffa05,demkowicz-cao05,demkowicz08,Melenk_Rojik}, which were restricted to spatial dimensions $d \in \{1,2,3\}$, are of independent interest. Third, on a more technical level, we remove the condition of \cite[D.2.4]{Angleitner_H_matrices_FEM} on the relation between
the minimal and the maximal mesh size and the maximal element (see \cref{SSec_Proof_overview} for details).

We follow the notation of \cite{Angleitner_H_matrices_FEM}. In particular, we write $a \cleq b$, if there exists a constant $C>0$, such that $a \leq Cb$. The constant $C$ may depend on the space dimension $d$, the problem domain $\Omega$, the PDE coefficients $a_1,a_2,a_3$, the shape-regularity constant $\CShape$, the admissibility constant $\CAdm$ or the sparsity constant $\CSparse$. However, it may \emph{not} depend on the polynomial degree $p$.

\section{Main results} \label{Sec_Main_results}
\subsection{The model problem} \label{SSec_Model_problem}

We investigate the following \emph{model problem}: Let $d \geq 1$ and $\Omega \subseteq \R^d$ be a bounded polyhedral Lipschitz domain. Furthermore, let $a_1 \in \Lp{\infty}{\Omega,\R^{d \times d}}$, $a_2 \in \Lp{\infty}{\Omega,\R^d}$ and $a_3 \in \Lp{\infty}{\Omega,\R}$ be given coefficient functions and $f \in \Lp{2}{\Omega}$ be a given right-hand side. We seek a weak solution $u \in \HkO{1}{\Omega}$ to the following equations:
\begin{equation*}
\begin{array}{rcll}
-\div{a_1 \gradN{u}} + a_2 \cdot \gradN{u} + a_3u &=& f & \text{in} \,\, \Omega, \\
u &=& 0 & \text{on} \,\, \boundaryN{\Omega}.
\end{array}
\end{equation*}

We assume that $a_1$ is coercive in the sense $\skalar{a_1(x) y}{y} \geq \alpha_1 \norm[2]{y}^2$
for all $x \in \Omega$, $y \in \R^d$ and some constant 
$\alpha_1 > \CPcre^2 (\norm[\Lp{\infty}{\Omega}]{a_2} + \norm[\Lp{\infty}{\Omega}]{a_3}) \geq 0$. 
Here, $\CPcre>0$ denotes the constant in the Poincar\'e inequality $\norm[\Hk{1}{\Omega}]{\cdot} \leq \CPcre \seminorm[\Hk{1}{\Omega}]{\cdot}$ on $\HkO{1}{\Omega}$.

\begin{definition} \label{Bilinear_form}
We introduce the following bilinear form:
\begin{equation*}
\forall u,v \in \HkO{1}{\Omega}: \quad \quad \bilinear[a]{u}{v} := \skalar[\Lp{2}{\Omega}]{a_1 \gradN{u}}{\gradN{v}} + \skalar[\Lp{2}{\Omega}]{a_2 \* \gradN{u}}{v} + \skalar[\Lp{2}{\Omega}]{a_3 u}{v}.
\end{equation*}
\end{definition}

The weak formulation of the \emph{model problem} reads as follows: Find $u \in \HkO{1}{\Omega}$ such that
\begin{equation*}
\forall v \in \HkO{1}{\Omega}: \quad \quad \bilinear[a]{u}{v} = \skalar[\Lp{2}{\Omega}]{f}{v}.
\end{equation*}

The assumptions on the PDE coefficients imply that the bilinear form $\bilinear[a]{\cdot}{\cdot}$ is continuous and coercive. In particular, the well-known Lax-Milgram Lemma yields the existence of a unique solution $u \in \HkO{1}{\Omega}$.

\subsection{The spline spaces}

For the discretization of the model problem, we introduce the well-known spline spaces $\SkpO{1}{p}{\Elements} \subseteq \HkO{1}{\Omega}$, where $\Elements$ is a \emph{mesh} on $\Omega$ and $p \geq 1$ is a prescribed polynomial degree. 

\begin{definition} \label{Mesh}
A finite set $\Elements \subseteq \Pow{\Omega}$ is a \emph{mesh}, if there exists an open simplex $\hat{T} \subseteq \R^d$ (the \emph{reference element}) such that every \emph{element} $T \in \Elements$ is of the form $T = F_T(\hat{T})$, where $\fDef{F_T}{\R^d}{\R^d}$ is an affine diffeomorphism. Furthermore, the elements must be pairwise disjoint, i.e., $\meas{T \cap S} = 0$ for all $T \neq S \in \Elements$, and constitute a partition of $\Omega$, i.e., $\bigcup_{T \in \Elements} \overline{T}  = \overline{\Omega}$. Finally, a mesh must be regular in the sense of \cite{Ciarlet_FEM_Meshes}, i.e., it does not contain any hanging nodes.
\end{definition}

For every element $T \in \Elements$, we define the \emph{patch} $\Patch{\Elements}{T} := \Set{S \in \Elements}{\closureN{S} \cap \closureN{T} \neq \emptyset}$. To measure the size of an element $T \in \Elements$, we introduce the local \emph{mesh width} $\h{T} := \sup_{x,y \in T} \norm[2]{y-x}$. Similarly, for every collection of elements $\ElementsB \subseteq \Elements$, we set $\hMax{\ElementsB} := \max_{T \in \ElementsB} \h{T}$ and $\hMin{\ElementsB} := \min_{T \in \ElementsB} \h{T}$. In the case $\ElementsB = \Elements$, we abbreviate $\hMax{} := \hMax{\Elements}$ and $\hMin{} := \hMin{\Elements}$.

For every $T \in \Elements$, we denote the center of the largest inscribable ball by $x_T \in T$ (the \emph{incenter}). Here, $\Ball[2]{x}{r} := \Set{y \in \R^d}{\norm[2]{y-x}<r}$ is the open ball with radius $r>0$, centered around $x \in \R^d$.

\begin{assumption} \label{Shape_regularity}
We assume that $\Elements$ is part of a \emph{shape-regular} family of meshes, i.e., there exists a constant $\CShape \geq 1$ such that
\begin{equation*}
\forall T \in \Elements: \quad \quad \Ball[2]{x_T}{\CShape^{-1} \h{T}} \subseteq T \subseteq \bigcup\Patch{\Elements}{T} \subseteq \Ball[2]{x_T}{\CShape \h{T}}.
\end{equation*}
\end{assumption}

Let us next give a formal definition of the spline spaces.

\begin{definition} \label{Space_Skp}
We set
\begin{equation*}
\SkpO{1}{p}{\Elements} := \Set{v \in \HkO{1}{\Omega}}{\forall T \in \Elements: v \circ F_T \in \Pp{p}{\hat{T}}},
\end{equation*}
where $\Pp{p}{\hat{T}} := \spanN{\Set{\hat{T} \ni x \mapsto x^q}{\norm[1]{q} \leq p}}$ denotes the usual space of polynomials of (total) degree $p$ on the reference element $\hat{T}$. Similarly, we set
\begin{equation*}
\Skp{0}{p}{\Elements} := \Set{v \in \Lp{2}{\Omega}}{\forall T \in \Elements: v \circ F_T \in \Pp{p}{\hat{T}}}.
\end{equation*}
\end{definition}

\begin{remark}
Note that the polynomial degree $p$ is the same for all elements of the mesh $\Elements$. In contrast, in the $hp$-version of the FEM, a polynomial degree distribution $\Set{p_T}{T \in \Elements}$ is prescribed. In this context, $p$ may be regarded as the maximum of these values. 
The analysis of a general polynomial degree distribution is beyond the scope of the present work, and we focus on the uniform polynomial degree distribution. 
\end{remark}

The following definition introduces the bases of $\SkpO{1}{p}{\Elements}$ that we consider.

\begin{definition} \label{Basis_fcts}
Let $N := N(\Elements,p) := \dimN{\SkpO{1}{p}{\Elements}}$ and let $\set{\phi_1,\dots,\phi_N} \subseteq \SkpO{1}{p}{\Elements}$ be a basis. We say that the basis \emph{allows for a system of local dual functions}, if there exist functions $\set{\lambda_1,\dots,\lambda_N} \subseteq \Lp{2}{\Omega}$ with the following properties:
\begin{enumerate}
\item Duality: For all $n,m \in \set{1,\dots,N}$, there holds $\skalar[\Lp{2}{\Omega}]{\phi_n}{\lambda_m} = \kronecker{nm}$ (Kronecker delta).

\item Stability: There exist constants $C_{\rm stab}$, $\CStab>0$ such that
\begin{equation*}
\forall \mvemph{x} \in \R^N: \quad \quad \normB[\Lp{2}{\Omega}]{\sum_{m=1}^{N} \mvemph{x}_m \lambda_m} \leq C_{\rm stab} p^{\CStab} \hMin{}^{-d/2} \norm[2]{\mvemph{x}}.
\end{equation*}

\item Locality and overlap: For every $n \in \set{1,\dots,N}$, there exists a characteristic element $T_n \in \Elements$ such that $\supp{\lambda_n} \subseteq \bigcup\Patch{\Elements}{T_n}$. For all $T \in \Elements$, there holds the bound $\cardN{\Set{n}{T_n = T}} \leq \binom{p+d}{d}$.

\end{enumerate}
\end{definition}

\begin{example}
Typically, a finite element basis $\set{\phi_1,\dots,\phi_N} \subseteq \SkpO{1}{p}{\Elements}$ is constructed from a predefined basis of \emph{shape functions} $\Set{\hat{\phi}_i}{i=1,\dots,\binom{d+p}{d}} \subseteq \Pp{p}{\hat{T}}$ on the reference element $\hat{T} \subseteq \R^d$. Following \cite[Sec.3.3]{Angleitner_H_matrices_FEM}, we can then build the dual functions $\set{\lambda_1,\dots,\lambda_N} \subseteq \Lp{2}{\Omega}$ from the \emph{dual shape functions} $\Set{\hat{\lambda}_j}{j=1,\dots,\binom{d+p}{d}} \subseteq \Pp{p}{\hat{T}}$, which are defined via the conditions $\skalar[\Lp{2}{\hat{T}}]{\hat{\phi}_i}{\hat{\lambda}_j} = \kronecker{ij}$. However, since we want to include the case $p \rightarrow \infty$ in our analysis, the standard Lagrange basis has to be replaced with a basis with good stability properties in $p$.

In $d=2$ space dimensions, for example, we can pick the shape functions $\hat{\phi}_i$ from \cite[D.2.4.]{Melenk_hp_bases}. 
It was shown in \cite[L.4.4.]{Melenk_hp_bases} that the corresponding coordinate mapping $\hat{\Phi}c := \sum_{i} c_i \hat{\phi}_i$ exhibits the stability bounds $p^{-3}\norm[2]{c} \cleq \norm[\Lp{2}{\hat{T}}]{\hat{\Phi}c} \cleq \norm[2]{c}$ for all $c \in \R^{(p+2)(p+1)/2}$. In particular, using the Euclidean unit vectors $e_i \in \R^{(p+2)(p+1)/2}$, we also get a stability bound for the dual shape functions $\hat{\lambda}_j$:
\begin{equation*}
\norm[\Lp{2}{\hat{T}}]{\hat{\lambda}_j}^2 = \skalar[\Lp{2}{\hat{T}}]{\hat{\Phi}\hat{\Phi}^{-1}\hat{\lambda}_j}{\hat{\lambda}_j} = \sum_i \skalar[2]{\hat{\Phi}^{-1}\hat{\lambda}_j}{e_i} \skalar[\Lp{2}{\hat{T}}]{\hat{\phi}_i}{\hat{\lambda}_j} = \skalar[2]{\hat{\Phi}^{-1}\hat{\lambda}_j}{e_j} \leq \norm[2]{\hat{\Phi}^{-1}\hat{\lambda}_j} \cleq p^3\norm[\Lp{2}{\hat{T}}]{\hat{\lambda}_j},
\end{equation*}
so that $\norm[\Lp{2}{\hat{T}}]{\hat{\lambda}_j} \cleq p^3$. Updating the proof of \cite[L.3.6]{Angleitner_H_matrices_FEM}, we find that the stability bound in \cref{Basis_fcts} is satisfied with $\CStab = d/2+3 = 4$.

Finally, let us motivate the assumption $\cardN{\Set{n}{T_n = T}} \leq \binom{p+d}{d}$ from \cref{Basis_fcts}: The previously mentioned construction in \cite[Sec.3.3]{Angleitner_H_matrices_FEM} guarantees that not only $\supp{\lambda_n} \subseteq \bigcup\Patch{\Elements}{T_n}$, but even $\supp{\lambda_n} = T_n$. Furthermore, owing to item $(1)$ in \cref{Basis_fcts}, the system $\set{\lambda_1,\dots,\lambda_N} \subseteq \Lp{2}{\Omega}$ is linearly independent. Then, given an arbitrary element $T \in \Elements$, the system $\Set{\restrictN{\lambda_n}{T}}{n \in \set{1,\dots,N} \,\,\text{with}\,\, T_n = T} \subseteq \Pp{p}{T}$ must be linearly independent as well. It follows that $\cardN{\Set{n}{T_n = T}} \leq \dim{\Pp{p}{T}} = \binom{d+p}{d}$, i.e., the overlap condition is fulfilled.
\end{example}

The supports of the dual functions play an import role in our analysis. It pays off to introduce some names:
\begin{definition} \label{Dual_supports}
For all $n \in \set{1,\dots,N}$ and all $I \subseteq \set{1,\dots,N}$, we set
\begin{equation*}
\Omega_n := \supp{\lambda_n} \subseteq \R^d, \quad \quad \quad \Omega_I := \bigcup_{n \in I} \Omega_n \subseteq \R^d.
\end{equation*}
\end{definition}

\subsection{The system matrix} \label{SSec_System_matrix}

Now that the spline spaces $\SkpO{1}{p}{\Elements} \subseteq \HkO{1}{\Omega}$ are at our disposal, the \emph{discrete model problem} reads as follows: For given $f \in \Lp{2}{\Omega}$, find $u \in \SkpO{1}{p}{\Elements}$ such that
\begin{equation*}
\forall v \in \SkpO{1}{p}{\Elements}: \quad \quad \bilinear[a]{u}{v} = \skalar[\Lp{2}{\Omega}]{f}{v}.
\end{equation*}

Again, existence and uniqueness of a solution $u \in \SkpO{1}{p}{\Elements}$ follow from the Lax-Milgram Lemma.

As usual, given a basis of the ansatz space, the discrete model problem can be rephrased as an equivalent linear system of equations. The bilinear form $\bilinear[a]{\cdot}{\cdot}$ from \cref{Bilinear_form} and the basis functions $\phi_n \in \SkpO{1}{p}{\Elements}$ from \cref{Basis_fcts} compose the governing system matrix.

\begin{definition} \label{System_matrix}
We define the system matrix
\begin{equation*}
\mvemph{A} := (\bilinear[a]{\phi_n}{\phi_m})_{m,n=1}^{N} \in \R^{N \times N}.
\end{equation*}
\end{definition}

Note that the unique solvability of the discrete model problem already ensures that the matrix $\mvemph{A}$ is invertible.

\subsection{Hierarchical matrices} \label{SSec_Hierarchical_matrices}

In this section, we provide the basic definitions from the theory of hierarchical matrices. We slightly divert from \cite[Section 2.5]{Angleitner_H_matrices_FEM} and use the formulation from our previous work on radial basis functions, \cite[Section 2.4]{Angleitner_H_matrices_RBF}. As will be discussed later in \cref{SSec_Proof_overview}, a formulation in terms of axes-parallel boxes $B \subseteq \R^d$, rather than collections of elements $\ElementsB \subseteq \Elements$, has certain advantages. An extensive discussion of hierarchical matrices can be found, e.g., in the books \cite{Hackbusch_Hierarchical_matrices,bebendorfbook,boermbook}.

\begin{definition} \label{Box}
A subset $B \subseteq \R^d$ is called \emph{(axes-parallel) box}, if it has the form $B = \bigtimes_{i=1}^{d} (a_i,b_i)$ with $a_i < b_i$.
\end{definition}

For the next definition, we remind the reader of the subsets $\Omega_I \subseteq \R^d$, introduced in \cref{Dual_supports}. Furthermore, we use the usual definition of Euclidean diameter and distance of subsets $B,B_1,B_2 \subseteq \R^d$, i.e.,
\begin{equation*}
\diam[2]{B} := \sup_{x,y \in B} \norm[2]{x-y}, \quad \quad \quad \dist[2]{B_1}{B_2} := \inf_{x \in B_1, y \in B_2} \norm[2]{x-y}.
\end{equation*}

\begin{definition} \label{Block_partition}
Let $\CSmall,\CAdm>0$. A tuple $(I,J)$ with $I,J \subseteq \set{1,\dots,N}$ is called \emph{small}, if there holds $\min\set{\cardN{I}, \cardN{J}} \leq \CSmall$. It is called \emph{admissible}, if there exist boxes $B_I,B_J \subseteq \R^d$ such that $\Omega_I \subseteq B_I$, $\Omega_J \subseteq B_J$ and
\begin{equation*}
\diam[2]{B_I} \leq \CAdm \dist[2]{B_I}{B_J}.
\end{equation*}

A set $\BPart$ of tuples $(I,J)$ with $I,J \subseteq \set{1,\dots,N}$ is called \emph{sparse hierarchical block partition}, if the following assumptions are satisfied:
\begin{enumerate}
\item The system $\Set{I \times J}{(I,J) \in \BPart}$ forms a partition of $\set{1,\dots,N} \times \set{1,\dots,N}$.

\item There holds $\BPart = \BPartSmall \cup \BPartAdm$, where every $(I,J) \in \BPartSmall$ is small and every $(I,J) \in \BPartAdm$ is admissible.

\item For all $\mvemph{B} \in \R^{N \times N}$, there holds the bound
\begin{equation*}
\norm[2]{\mvemph{B}} \cleq \ln(\hMin{}^{-d}) \max_{(I,J) \in \BPart} \norm[2]{\restrictN{\mvemph{B}}{I \times J}}.
\end{equation*}

\end{enumerate}
\end{definition}

\begin{definition} \label{H_matrices}
Let $\BPart$ be a sparse hierarchical block partition and $r \in \N$ a given \emph{block rank bound}. We define the set of \emph{$\mathcal{H}$-matrices} by
\begin{equation*}
\HMatrices{\BPart}{r} := \Set{\mvemph{B} \in \R^{N \times N}}{\forall (I,J) \in \BPartAdm: \exists \mvemph{X} \in \R^{I \times r}, \mvemph{Y} \in \R^{J \times r}: \restrictN{\mvemph{B}}{I \times J} = \mvemph{X} \mvemph{Y}^T}.
\end{equation*}
\end{definition}

We mention that a sparse hierarchical block partition $\BPart$ can be constructed, e.g., using the \emph{geometrically balanced clustering strategy} from \cite{Hackbusch_Geometric_clustering}. In fact, recall from item $(3)$ of \cref{Basis_fcts} that, at any given point $x \in \R^d$, no more than $\binom{p+d}{d}$ of the sets $\Omega_n$ can overlap. Then, assuming that the clustering parameter $\CSmall$ is chosen such that $\binom{p+d}{d} \leq \CSmall \cleq \binom{p+d}{d}$, the authors of \cite{Hackbusch_Geometric_clustering} derived the following bounds for the \emph{block cluster tree $\Tree{N \times N}$}:
\begin{equation*}
C_{\mathrm{sparse}}(\Tree{N \times N}) \cleq 1, \quad \quad \quad \depth{\Tree{N \times N}} \cleq \ln(\hMin{}^{-d}).
\end{equation*}
(See, e.g., \cite{Hackbusch_Geometric_clustering} or \cite{Hackbusch_Hierarchical_matrices} for a precise definition of these fundamental quantities.) The asserted bound in item $(3)$ of \cref{Block_partition} then follows readily from \cite[L.6.5.8]{Hackbusch_Hierarchical_matrices}.

Finally, according to \cite[Lemma 6.3.6]{Hackbusch_Hierarchical_matrices}, the memory requirements to store an $\mathcal{H}$-matrix $\mvemph{B} \in \HMatrices{\BPart}{r}$ can be bounded by
\begin{equation*}
N_{\mathrm{memory}} \leq C_{\mathrm{sparse}}(\Tree{N \times N}) (r+\CSmall) \depth{\Tree{N \times N}} N \cleq (r+p^d) \ln(\hMin{}^{-d}) N.
\end{equation*}

Since $\mvemph{B}$ shall serve as an approximation for the $N^2$ entries of the matrix $\mvemph{A}^{-1} \in \R^{N \times N}$, this approach requires bounds of $r$, $p$ and $\hMin{}$ in terms of $N = \dimN{\SkpO{1}{p}{\Elements}} \ceq p^d \cardN{\Elements}$. For example, if the mesh $\Elements$ is such that
\begin{equation} \label{Mesh_Assumption}
1 \cleq (\cardN{\Elements})^{\CCard} \hMin{}^d,
\end{equation}
for some constant $\CCard \geq 1$, then we end up with the following bound:
\begin{equation*}
N_{\text{memory}} \cleq (r+p^d)\ln(N/p^d) N \leq (r+p^d) \ln(N) N.
\end{equation*}

\subsection{The main result} \label{SSec_The_main_result}

The following theorem is the main result of the present work. Roughly speaking, it states that inverses of FEM matrices can be approximated at an exponential rate in the block rank by hierarchical matrices.

\begin{theorem} \label{Main_result}
Let $\bilinear[a]{\cdot}{\cdot}$ be the elliptic bilinear form from \cref{Bilinear_form}, let $\Elements \subseteq \Pow{\Omega}$ be a mesh as in \cref{Mesh}, and let $p \geq 1$ be an arbitrary integer. Let $\set{\phi_1,\dots,\phi_N} \subseteq \SkpO{1}{p}{\Elements}$ be a basis that allows for a system of local dual functions (see \cref{Basis_fcts}) and denote the corresponding stability constant by $\CStab>0$. Furthermore, let $\mvemph{A} \in \R^{N \times N}$ be the Galerkin stiffness matrix from \cref{System_matrix} and $\BPart$ be a sparse hierarchical block partition as in \cref{Block_partition}. Finally, let $\CRed \geq 2$ be the constant from \cref{Melenk_Rojik} further below. Then, there exists a constant $\CExp = C(d,\Omega,a,\CShape,\CAdm) > 0$ such that the following holds true: For every block rank bound $r \in \N$, there exists an $\mathcal{H}$-matrix $\mvemph{B} \in \HMatrices{\BPart}{r}$ such that
\begin{equation*}
\norm[2]{\mvemph{A}^{-1} - \mvemph{B}} \cleq p^{2\CStab} \ln(\hMin{}^{-d}) \hMin{}^{-d} \exp(-\CExp r^{1/(d+1)} p^{-\CRed}).
\end{equation*}
\end{theorem}

Note that, apart from shape-regularity, this result needs no further assumptions on the mesh $\Elements$. However, the previous discussion about the storage complexity of $\mathcal{H}$-matrices suggests that we might as well assume \cref{Mesh_Assumption}. In this case, we immediately get the following corollary:

\begin{corollary} \label{Main_result_corollary}
Assume that the mesh $\Elements$ satisfies \cref{Mesh_Assumption}, for some constant $\CCard \geq 1$. Then \cref{Main_result} holds verbatim with a bound
\begin{equation*}
\norm[2]{\mvemph{A}^{-1} - \mvemph{B}} \cleq p^{2\CStab-d\CCard} \ln(N) N^{\CCard} \exp(-\CExp r^{1/(d+1)} p^{-\CRed}).
\end{equation*}

\end{corollary}

\begin{figure}[H]
\begin{center}
\includegraphics[width=1\textwidth, trim=0cm 0cm 0cm 0cm]{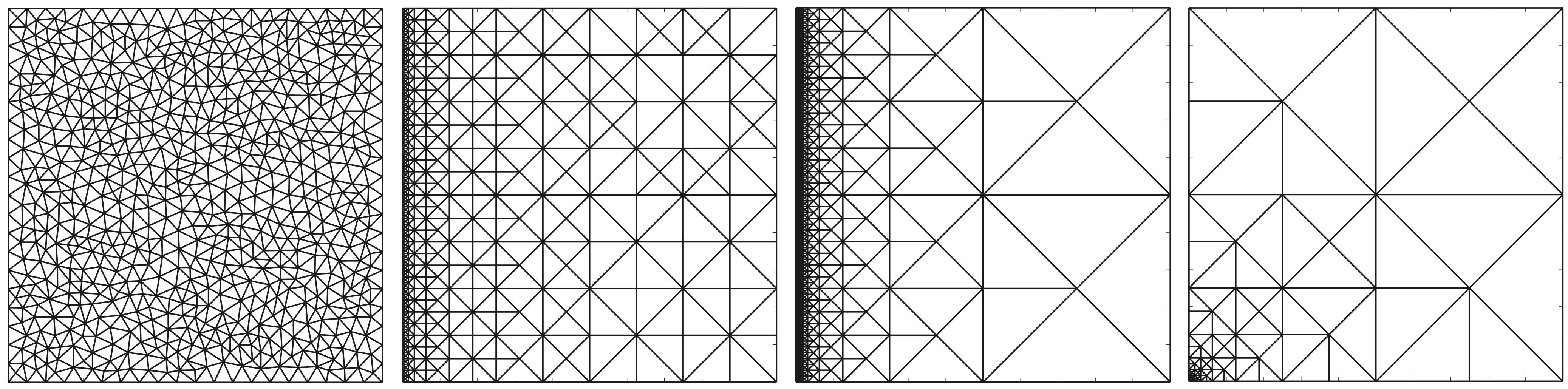}
\caption{From left to right: Uniform, algebraically graded towards edge, exponentially graded towards edge, exponentially graded towards corner. \cref{Main_result_corollary} covers the first, second and third type, but \emph{not} the last one.}
\label{Mesh_examples}
\end{center}
\end{figure}

The assumption \cref{Mesh_Assumption} is satisfied for a wide variety of meshes including uniform, algebraically graded and even some exponentially graded ones (cf. \cref{Mesh_examples}). Given parameters $H>0$ and $\alpha \in [1,\infty]$, and a subset $\Gamma \subseteq \Omega$, we say that a mesh $\Elements$ is \emph{graded towards $\Gamma$}, if there holds the relationship $\h{T} \ceq \dist[2]{x_T}{\Gamma}^{1-1/\alpha} H$, for all $T \in \Elements$. The case $\alpha=1$ is called \emph{uniform}, the case $\alpha \in (1,\infty)$ is an \emph{algebraic grading} and the case $\alpha=\infty$ represents \emph{exponential grading}. If $\alpha \in [1,\infty)$, then \cref{Mesh_Assumption} is satisfied with $\CCard = \alpha$. In the case $\alpha=\infty$, however, the relationship need not necessarily be fulfilled.

\begin{remark}
One possible application of exponentially graded meshes can be found in the context of the \emph{boundary concentrated FEM}, e.g., \cite{Melenk_Boundary_concentrated_FEM} and \cite{Melenk_Boundary_concentrated_FEM_Direct_solver}. This method is similar to the boundary element method (BEM), in that most mesh elements are near the boundary of $\Omega$. However, we mention that \cref{Main_result} is not directly applicable to this method, because \cite{Melenk_Boundary_concentrated_FEM} replaces the (constant-degree) spline spaces $\SkpO{1}{p}{\Elements}$ from \cref{Space_Skp} with variable degree spline spaces $\SkpO{1}{\mvemph{p}}{\Elements}$, $\mvemph{p} = \Set{p_T}{T \in \Elements}$.
\end{remark}

\begin{remark}
In contrast to our previous work, \cite[T.2.15]{Angleitner_H_matrices_FEM}, the constant $\CCard$ from \cref{Main_result_corollary} does not enter the argument of the exponential in the error bound any more. In particular, the rate of convergence (as $r \rightarrow \infty$) does not deteriorate for meshes with stronger grading. This behavior is in accordance with the radial basis function setting, observed in \cite[T.2.18]{Angleitner_H_matrices_RBF}.
\end{remark}

\section{Proof of main result} \label{Sec_Proof_of_main_result}
\subsection{Overview} \label{SSec_Proof_overview}

The main result of our previous work, \cite[T.2.15]{Angleitner_H_matrices_FEM}, was applicable to a class of meshes with \emph{locally bounded cardinality}, which included uniform and algebraically graded meshes, but excluded exponential grading. Quoting \cite[D.2.4]{Angleitner_H_matrices_FEM}, a mesh $\Elements$ has locally bounded cardinality, if there exists a constant $\CCard \geq 1$ such that
\begin{equation}
\label{eq:locally-bounded-cardinality}
\hMax{}^{\CCard} \cleq \hMin{}, \quad \quad \quad \forall \ElementsB \subseteq \Elements: \quad \cardN{\ElementsB} \cleq (1 + \diam[\Elements]{\ElementsB}/\hMax{\ElementsB})^{d\CCard}.
\end{equation}

The left part of (\ref{eq:locally-bounded-cardinality}) could easily be replaced by the assumption $1 \cleq N^{\CCard} \hMin{}^d$ from \cref{Main_result_corollary}. However, we may as well avoid it altogether. In fact, whenever $\hMin{}$ appears during the subsequent proof, we simply leave it as is and refrain from replacing it with any potential lower bound. Consequently, the error bound in \cref{Main_result} is formulated in terms of $\hMin{}$, rather than $\hMax{}$ or $N$.

The right part of the assumption (\ref{eq:locally-bounded-cardinality}) is much harder to remove. It was used in the proof of \cite[T.3.31]{Angleitner_H_matrices_FEM} to find a suitable rank bound for the \emph{single-step coarsening operator} $\fDef{\CoarseningOp{\ElementsB}{\delta}}{\SHarm{\inflateN{\ElementsB}{\delta}}}{\SHarm{\ElementsB}}$, where $\ElementsB \subseteq \Elements$ was a given set of elements and where $\delta>0$ was the inflation radius. (See \cite[D.3.25]{Angleitner_H_matrices_FEM} for the precise definition of the space $\SHarm{\ElementsB}$.) Let us briefly illustrate why our previous construction of this operator might fail for exponentially graded meshes: For certain technical reasons, the cases $\delta \cgeq \hMax{\ElementsB}$ and $\delta \cleq \hMax{\ElementsB}$ were treated differently. In the case $\delta \cgeq \hMax{\ElementsB}$, we used a uniform mesh $\ElementsS$ of meshsize $\delta$ for re-interpolation, producing an approximant with roughly $\Landau{\delta^{-d}}$ degrees of freedom. In the remaining case $\delta \cleq \hMax{\ElementsB}$, however, re-interpolation was not necessary. In fact, due to the assumption of locally bounded cardinality, the function $u \in \SHarm{\inflateN{\ElementsB}{\delta}}$ to be approximated had less than $\Landau{\delta^{-d\CCard}}$ degrees of freedom anyways. Now, in the case of an exponentially graded mesh $\Elements$, the input $u$ might have significantly more than $\Landau{\delta^{-d\CCard}}$ degrees of freedom. After all, we could refine one of the elements in $\ElementsB$ arbitrarily often without ever affecting $\hMax{\ElementsB}$, essentially raising the dimension above any fixed power of $\delta^{-1}$.

The main idea of our revised proof is to eliminate all occurrences of the technical assumption $\delta \cgeq \hMax{\ElementsB}$, so that the uniform mesh $\ElementsS$ can be used in all cases, regardless of the relative sizes of $\delta$ and $\hMax{\ElementsB}$. This problematic assumption rooted in our decision to use a \emph{discrete} cut-off function $\CutoffFc{\ElementsB}{\delta} \in \Skp{1}{1}{\Elements}$ with $\restrictN{\CutoffFc{\ElementsB}{\delta}}{\ElementsB} \equiv 1$ and $\supp[\Elements]{\CutoffFc{\ElementsB}{\delta}} \subseteq \inflateN{\ElementsB}{\delta}$. With its help, we proved the discrete Caccioppoli inequality $\delta\seminorm[\Hk{1}{\ElementsB}]{u} \cleq \norm[\Lp{2}{\inflateN{\ElementsB}{\delta}}]{u}$, $u \in \SHarm{\inflateN{\ElementsB}{\delta}}$, and constructed the discrete cut-off operator $\fDef{\CutoffOp{\ElementsB}{\delta}}{\Skp{1}{p}{\Elements}}{\Skp{1}{p}{\Elements}}$. In order to avoid the problems that come with \emph{discrete} cut-off functions, we revert to the original idea of \cite{Faustmann_H_matrices_FEM} of using axes-parallel boxes $B \subseteq \R^d$ instead of element clusters $\ElementsB \subseteq \Elements$. A \emph{smooth} cut-off function $\CutoffFc{B}{\delta} \in \Ck{\infty}{\closureN{\Omega}}$ with $\restrictN{\CutoffFc{B}{\delta}}{B} \equiv 1$ and $\supp{\CutoffFc{B}{\delta}} \subseteq \inflateN{B}{\delta}$ can easily be constructed, even if $\delta \ll \hMax{\ElementsB}$. However, the cut-off operator now maps $\fDef{\CutoffOp{B}{\delta}}{\Hk{1}{\Omega}}{\Hk{1}{\Omega}}$, so that the previous definition of the space $\SHarm{\ElementsB} \subseteq \SkpO{1}{p}{\Elements}$ needs some minor modifications (recall from \cite[L.3.26]{Angleitner_H_matrices_FEM} that $\CutoffOp{\ElementsB}{\delta} u \in \SHarm{\ElementsB}$, for all $u \in \SHarm{\ElementsB}$, was an important property). Finally, we need to show the discrete Caccioppoli inequality for the updated spaces $\SHarm{B} \subseteq \HkO{1}{\Omega}$. It turns out that the assumption $\delta \cgeq \hMax{\ElementsB}$ can be dropped, because the discrete Caccioppoli inequality reduces to an inverse inequality on large elements.

\subsection{Reduction from matrix level to function space level} \label{SSec_Reduction_mat_lvl_fct_lvl}

\begin{definition} \label{Sol_op_disc}
Let $\fDef{a}{\HkO{1}{\Omega} \times \HkO{1}{\Omega}}{\R}$ be the bilinear form from \cref{Bilinear_form}. For every $f \in \Lp{2}{\Omega}$, denote by $\SolutionOp{\Elements} f \in \SkpO{1}{p}{\Elements}$ the unique function satisfying the following variational equality:
\begin{equation*}
\forall v \in \SkpO{1}{p}{\Elements}: \quad \quad \bilinear[a]{\SolutionOp{\Elements} f}{v} = \skalar[\Lp{2}{\Omega}]{f}{v}.
\end{equation*}

The linear mapping $\fDef{\SolutionOp{\Elements}}{\Lp{2}{\Omega}}{\SkpO{1}{p}{\Elements}}$ is called \emph{discrete solution operator}.
\end{definition}

Note that existence and uniqueness of $\SolutionOp{\Elements} f$ are provided by the Lax-Milgram Lemma. Additionally, there holds the a priori bound $\norm[\Hk{1}{\Omega}]{\SolutionOp{\Elements} f} \cleq \norm[\Lp{2}{\Omega}]{f}$.

According to \cref{H_matrices} and the asserted stability bound in \cref{Block_partition}, the task of approximating the whole matrix $\mvemph{A}^{-1}$ by an $\mathcal{H}$-matrix $\mvemph{B} \in \HMatrices{\BPart}{r}$ reduces to the one of approximating the admissible blocks $\restrictN{\mvemph{A}^{-1}}{I \times J}$ by means of matrices $\mvemph{X} \in \R^{I \times r}$ and $\mvemph{Y} \in \R^{J \times r}$. We then proceed as in \cite{Angleitner_H_matrices_FEM} and transfer from the matrix level to the function space level:

\begin{lemma} \label{Mat_lvl_to_fct_lvl}
Let $I,J \subseteq \set{1,\dots,N}$ and $V \subseteq \Lp{2}{\Omega}$ be a finite-dimensional subspace. Then, there exist an integer $r \leq \dimN{V}$ and matrices $\mvemph{X} \in \R^{I \times r}$ and $\mvemph{Y} \in \R^{J \times r}$, such that there holds the following error bound:
\begin{equation*}
\norm[2]{\restrictN{\mvemph{A}^{-1}}{I \times J} - \mvemph{X}\mvemph{Y}^T} \cleq p^{2\CStab} \hMin{}^{-d} \sup_{\substack{f \in V: \\ \supp{f} \subseteq \Omega_J}} \inf_{v \in V} \frac{\norm[\Lp{2}{\Omega_I}]{\SolutionOp{\Elements} f - v}}{\norm[\Lp{2}{\Omega}]{f}}.
\end{equation*}
\end{lemma}

\begin{proof}
In \cite[L.3.13]{Angleitner_H_matrices_FEM}, the bound $\norm[2]{\restrictN{\mvemph{A}^{-1}}{I \times J} - \mvemph{X}\mvemph{Y}^T} \leq \norm{\Lambda}^2 \sup_f \inf_v \norm[\Lp{2}{\Omega_I}]{\SolutionOp{\Elements} f - v}/\norm[\Lp{2}{\Omega}]{f}$ was shown. Using the asserted stability bound from \cite[D.2.6]{Angleitner_H_matrices_FEM}, the constant $\norm{\Lambda}^2$ was subsequently bounded by $C(d,\CShape,p)\hMin{}^{-d}$. Since the stability bound in \cref{Basis_fcts} is now explicit in $p$, we can plug in $C(d,\CShape,p) = C(d,\CShape) p^{2\CStab}$.
\end{proof}

\subsection{The cut-off operator}

\begin{definition} \label{Box_infl}
Let $B = \bigtimes_{i=1}^{d} (a_i,b_i)$ with $a_i < b_i$ be a box as in \cref{Box}. For every $\delta \geq 0$, we introduce the \emph{inflated box} $\inflateN{B}{\delta} := \bigtimes_{i=1}^{d} (-\delta+a_i,b_i+\delta) \subseteq \R^d$.
\end{definition}

Note that $\inflateN{B}{\delta}$ is again a box. In particular, we can iterate $\inflate{\inflateN{B}{\delta}}{\delta} = \inflateN{B}{2\delta}$, $\inflate{\inflate{\inflateN{B}{\delta}}{\delta}}{\delta} = \inflateN{B}{3\delta}$, et cetera.

\begin{lemma} \label{Cut_off_fct}
Let $B \subseteq \R^d$ be a box and $\delta>0$. Then, there exists a \emph{smooth cut-off function} $\CutoffFc{B}{\delta}$ with the following properties:
\begin{equation*}
\CutoffFc{B}{\delta} \in \Ck{\infty}{\closureN{\Omega}}, \quad \quad \supp{\CutoffFc{B}{\delta}} \subseteq \Omega \cap \inflateN{B}{\delta}, \quad \quad \restrictN{\CutoffFc{B}{\delta}}{\Omega \cap B} \equiv 1, \quad \quad 0 \leq \CutoffFc{B}{\delta} \leq 1, \quad \quad \forall l \in \N_0: \seminorm[\Wkp{l}{\infty}{\Omega}]{\CutoffFc{B}{\delta}} \cleq \delta^{-l}.
\end{equation*}
\end{lemma}

\begin{proof}
Write $B = \bigtimes_{i=1}^{d} (a_i,b_i)$ and pick a univariate function $g \in \Ck{\infty}{\R}$ with $0 \leq g \leq 1$, $\restrictN{g}{(-\infty,0]} \equiv 1$ and $\restrictN{g}{[1/2,\infty)} \equiv 0$. Then, the function $\CutoffFc{B}{\delta}(x) := \prod_{i=1}^{d} g((a_i-x_i)/\delta)g((x_i-b_i)/\delta)$, $x \in \Omega$, is a valid choice.
\end{proof}

Since $\CutoffFc{B}{\delta}$ is now a smooth function, the corresponding cut-off operator has different mapping properties than before (cf. \cite[D.3.23]{Angleitner_H_matrices_FEM}).

\begin{definition} \label{Cut_off_op}
Let $B \subseteq \R^d$ be a box and $\delta>0$. Denote by $\CutoffFc{B}{\delta} \in \Ck{\infty}{\closureN{\Omega}}$ the smooth cut-off function from \cref{Cut_off_fct}. We define the \emph{cut-off operator}
\begin{equation*}
\fDefB[\CutoffOp{B}{\delta}]{\Hk{1}{\Omega}}{\Hk{1}{\Omega}}{v}{\CutoffFc{B}{\delta}v}.
\end{equation*}
\end{definition}

Let us summarize the key properties of this operator:
\begin{lemma} \label{Cut_off_op_Props}
Let $B \subseteq \R^d$ be a box and $\delta>0$. For all $v \in \Hk{1}{\Omega}$, there hold the cut-off property $\supp{\CutoffOp{B}{\delta} v} \subseteq \Omega \cap \inflateN{B}{\delta}$ and the local projection property $\restrict{\CutoffOp{B}{\delta} v}{\Omega \cap B} = \restrictN{v}{\Omega \cap B}$. If $v \in \HkO{1}{\Omega}$, then $\CutoffOp{B}{\delta} v \in \HkO{1}{\Omega}$ as well. Finally, for all $v \in \Hk{1}{\Omega}$, there holds the stability estimate
\begin{equation*}
\sum_{l=0}^{1} \delta^l \seminorm[\Hk{l}{\Omega}]{\CutoffOp{B}{\delta} v} \cleq \sum_{l=0}^{1} \delta^l \seminorm[\Hk{l}{\Omega \cap \inflateN{B}{\delta}}]{v}.
\end{equation*}
\end{lemma}

\begin{proof}
The stability estimate follows from Leibniz' product rule for derivatives and the relation $\seminorm[\Wkp{l}{\infty}{\Omega}]{\CutoffFc{B}{\delta}} \cleq \delta^{-l}$, $l \in \N_0$. The remaining properties are immediate consequences of \cref{Cut_off_fct}.
\end{proof}

\subsection{The spaces of discrete and harmonic functions}

In this section, we introduce the spaces $\SHarm{B}$ of functions that are discrete and harmonic on some subset $B \subseteq \R^d$. The definition is slightly different from the previous one, \cite[D.3.25]{Angleitner_H_matrices_FEM}. Most notably, $\SHarm{B}$ is now an infinite-dimensional space.

\begin{definition} \label{Space_SHarm}
Let $B \subseteq \R^d$. A function $u \in \HkO{1}{\Omega}$ is called \dots
\begin{enumerate}
\item \dots \emph{discrete on $B$}, if there exists a function $\tilde{u} \in \SkpO{1}{p}{\Elements}$ such that $\restrictN{u}{\Omega \cap B} = \restrictN{\tilde{u}}{\Omega \cap B}$.
\item \dots \emph{harmonic on $B$}, if $\bilinear[a]{u}{v} = 0$ for all $v \in \SkpO{1}{p}{\Elements}$ with $\supp{v} \subseteq B$.
\end{enumerate}

We define the space of \emph{discrete and harmonic functions},
\begin{equation*}
\SHarm{B} := \Set{u \in \HkO{1}{\Omega}}{u \,\, \text{is discrete and harmonic on} \,\, B} \subseteq \HkO{1}{\Omega}.
\end{equation*}
\end{definition}

Note that $\SHarm{B}$ consists of global functions $\fDef{u}{\Omega}{\R}$ that merely happen to have some additional properties on the subset $\Omega \cap B$. Furthermore, we emphasize that $\SHarm{B}$ is an infinite-dimensional space, in general.

The next lemma summarizes the relevant properties of these spaces. Recall the definition of the discrete solution operator $\fDef{\SolutionOp{\Elements}}{\Lp{2}{\Omega}}{\SkpO{1}{p}{\Elements}}$ from \cref{Sol_op_disc} and the cut-off operator $\fDef{\CutoffOp{B}{\delta}}{\Hk{1}{\Omega}}{\Hk{1}{\Omega}}$ from \cref{Cut_off_op}.

\begin{lemma} \label{Space_SHarm_Props}
\begin{enumerate}
\item The subspace $\SHarm{B} \subseteq \HkO{1}{\Omega}$ is closed.
\item For all $B \subseteq B^+ \subseteq \R^d$, there holds $\SHarm{B^+} \subseteq \SHarm{B}$.
\item For all $B,D \subseteq \R^d$ with $B \cap D = \emptyset$ and all $f \in \Lp{2}{\Omega}$ with $\supp{f} \subseteq D$, there holds $\SolutionOp{\Elements} f \in \SHarm{B}$.
\item For all boxes $B \subseteq \R^d$, $\delta>0$ and $u \in \SHarm{B}$, there holds $\CutoffOp{B}{\delta} u \in \SHarm{B}$.
\end{enumerate}
\end{lemma}

\begin{proof}
We only show closedness. Since $\Omega \cap B \subseteq \R^d$ is open, the Sobolev space $\Hk{1}{\Omega \cap B}$ is well-defined. The subset $Z := \Set{\restrictN{\tilde{u}}{\Omega \cap B}}{\tilde{u} \in \SkpO{1}{p}{\Elements}} \subseteq \Hk{1}{\Omega \cap B}$ is a finite-dimensional subspace and thus closed. Note that any given function $u \in \HkO{1}{\Omega}$ is discrete on $B$ (in the sense of \cref{Space_SHarm}), if and only if $\restrictN{u}{\Omega \cap B} \in Z$.

Now, let $(u_n)_{n \in \N} \subseteq \SHarm{B}$ and $u \in \HkO{1}{\Omega}$ with $\norm[\Hk{1}{\Omega}]{u-u_n} \xrightarrow{n} 0$. In particular, for every $n \in \N$, we know that $\restrictN{u_n}{\Omega \cap B} \in Z$ and that $\bilinear[a]{u_n}{v} = 0$, for all $v \in \SkpO{1}{p}{\Elements}$ with $\supp{v} \subseteq B$. The trivial bound $\norm[\Hk{1}{\Omega \cap B}]{u-u_n} \leq \norm[\Hk{1}{\Omega}]{u-u_n} \xrightarrow{n} 0$ and the closedness of $Z$ immediately yield $\restrictN{u}{\Omega \cap B} \in Z$, meaning that $u$ is discrete on $B$. Finally, for all $v \in \SkpO{1}{p}{\Elements}$ with $\supp{v} \subseteq B$, we have
\begin{equation*}
\abs{\bilinear[a]{u}{v}} = \abs{\bilinear[a]{u-u_n}{v}} \cleq \norm[\Hk{1}{\Omega}]{u-u_n} \norm[\Hk{1}{\Omega}]{v} \xrightarrow{n} 0,
\end{equation*}
indicating that $u$ is harmonic on $B$. This concludes the proof of closedness.

\end{proof}

In the remainder of this section, we develop an improved version of the discrete Caccioppoli inequality from \cite[L.3.27]{Angleitner_H_matrices_FEM}. This time we are interested in large polynomial degrees $p \rightarrow \infty$ as well. Therefore, we need to revisit our previous proof and keep track of $p$. Since the elementwise Lagrange interpolant $\fDef{I_{\Elements}^p}{\CkPw{0}{\Elements}}{\Skp{0}{p}{\Elements}}$ from \cite[D.3.19]{Angleitner_H_matrices_FEM} is not suitable for large $p$, we employ an alternative operator:

\begin{lemma} \label{Melenk_Rojik}
There exists a linear operator $\fDef{J_{\Elements}^p}{\Skp{0}{p+2}{\Elements}}{\Skp{0}{p}{\Elements}}$ with the following properties:
\begin{enumerate}
\item \emph{Continuity and boundary values}: For all $v \in \SkpO{1}{p+2}{\Elements}$, there holds $J_{\Elements}^p v \in \SkpO{1}{p}{\Elements}$.

\item \emph{Supports}: For $v \in \Skp{0}{p+2}{\Elements}$, there holds $\supp{J_{\Elements}^p v} \subseteq \supp{v}$.

\item \emph{Error bound}: Let $\CRed := d(d+1)/4+2$. For all $\kappa \in \Skp{0}{1}{\Elements}$, all $u \in \Skp{0}{p}{\Elements}$ and all $T \in \Elements$, there holds the error bound
\begin{equation*}
\sum_{l=0}^{1} \h{T}^l \seminorm[\Hk{l}{T}]{(\identity-J_{\Elements}^p)(\kappa^2 u)} \cleq p^{\CRed} \h{T} \seminorm[\Wkp{1}{\infty}{T}]{\kappa^2} \norm[\Lp{2}{T}]{u}.
\end{equation*}

\end{enumerate}
\end{lemma}

For the sake of readability, we postpone the lengthy proof of \cref{Melenk_Rojik} to \cref{Sec_Rojik_proof} further below. Furthermore, we mention that the value of the constant $\CRed$ is not optimal. (The subscript ``red'' is reminiscent of the fact that the the operator $J_{\Elements}^p$ \emph{red}uces the polynomial degree of its input.)

For the subsequent revision of \cite[L.3.27]{Angleitner_H_matrices_FEM}, we remind the reader of our definition of \emph{inflated clusters}:
\begin{equation*}
\forall \ElementsB \subseteq \Elements: \forall \delta > 0: \quad \quad \inflateN{\ElementsB}{\delta} := \Set{T \in \Elements}{\exists S \in \ElementsB: \norm[2]{x_T-x_S} \leq \delta}.
\end{equation*}

\begin{lemma} \label{Space_SHarm_Cacc_old}
Let $\ElementsB \subseteq \Elements$ be a collection of elements and $\delta>0$ be a parameter satisfying $4\CShape^3\hMax{\ElementsB} \leq \delta \cleq 1$. Let $u \in \SkpO{1}{p}{\Elements}$ be a function that satisfies $\bilinear[a]{u}{v} = 0$, for all $v \in \SkpO{1}{p}{\Elements}$ with $\supp{v} \subseteq \bigcup\inflateN{\ElementsB}{\delta}$. Then, with the constant $\CRed \geq 2$ from \cref{Melenk_Rojik}, there holds the Caccioppoli inequality
\begin{equation*}
\delta\seminorm[\Hk{1}{\ElementsB}]{u} \cleq p^{\CRed} \norm[\Lp{2}{\inflateN{\ElementsB}{\delta}}]{u}.
\end{equation*}
\end{lemma}

\begin{proof}
According to \cite[L.3.18]{Angleitner_H_matrices_FEM}, the assumption $4\CShape^3\hMax{\ElementsB} \leq \delta \cleq 1$ allows us to construct a \emph{discrete cut-off function} $\kappa$ with the following properties:
\begin{equation*}
\kappa \in \Skp{1}{1}{\Elements}, \quad \quad \supp{\kappa} \subseteq \bigcup\inflateN{\ElementsB}{\delta}, \quad \quad \restrictN{\kappa}{\bigcup\ElementsB} \equiv 1, \quad \quad 0 \leq \kappa \leq 1, \quad \quad \forall l \in \set{0,1}: \seminorm[\Wkp{l}{\infty}{\Omega}]{\kappa} \cleq \delta^{-l}.
\end{equation*}

Let $u \in \SkpO{1}{p}{\Elements}$ as above. We consider the function $v := J_{\Elements}^p(\kappa^2 u)$, where $\fDef{J_{\Elements}^p}{\Skp{0}{p+2}{\Elements}}{\Skp{0}{p}{\Elements}}$ denotes the approximation operator from \cref{Melenk_Rojik}. Since $\kappa^2 u \in \SkpO{1}{p+2}{\Elements}$, we know that $v \in \SkpO{1}{p}{\Elements}$ and that $\supp{v} \subseteq \supp{\kappa^2 u} \subseteq \bigcup\inflateN{\ElementsB}{\delta}$. In particular, $v$ is a viable test function and we obtain the identity $\bilinear[a]{u}{v} = 0$. Using the constant $\CRed \geq 2$ defined in \cref{Melenk_Rojik}, we compute
\begin{equation*}
\begin{array}{rcccl}
\bilinear[a]{u}{\kappa^2 u} &=& \bilinear[a]{u}{\kappa^2 u - v} &=& \bilinear[a]{u}{(\identity-J_{\Elements}^p)(\kappa^2 u)} \\
&\stackrel{\cref{Bilinear_form}}{\cleq}& \sum\limits_{T \in \inflateN{\ElementsB}{\delta}} \norm[\Hk{1}{T}]{u} \norm[\Hk{1}{T}]{(\identity-J_{\Elements}^p)(\kappa^2 u)} &\stackrel{\cref{Melenk_Rojik}}{\cleq}& p^{\CRed} \sum\limits_{T \in \inflateN{\ElementsB}{\delta}} \seminorm[\Wkp{1}{\infty}{T}]{\kappa^2} \norm[\Hk{1}{T}]{u} \norm[\Lp{2}{T}]{u} \\
&\cleq& p^{\CRed} \delta^{-1} \sum\limits_{T \in \inflateN{\ElementsB}{\delta}} \norm[\Lp{\infty}{T}]{\kappa} \norm[\Hk{1}{T}]{u} \norm[\Lp{2}{T}]{u}.
\end{array}
\end{equation*}

Then, using Taylor's Theorem, a polynomial inverse inequality \cite{Ditzian_Inverse_inequality},  and the relation $\delta \cleq 1 \leq p$, we find that
\begin{equation*}
\norm[\Lp{\infty}{T}]{\kappa} \norm[\Hk{1}{T}]{u} \cleq \norm[\Lp{2}{T}]{u} + \(\min_{x \in T} \abs{\kappa(x)} + \h{T}\delta^{-1}\) \seminorm[\Hk{1}{T}]{u} \cleq p^2 \delta^{-1} \norm[\Lp{2}{T}]{u} + \norm[\Lp{2}{T}]{\kappa\gradN{u}},
\end{equation*}
which leads us to the following bound:
\begin{equation*}
\bilinear[a]{u}{\kappa^2 u} \cleq p^{\CRed} \delta^{-1} \sum_{T \in \inflateN{\ElementsB}{\delta}} \norm[\Lp{\infty}{T}]{\kappa} \norm[\Hk{1}{T}]{u} \norm[\Lp{2}{T}]{u} \cleq p^{\CRed+2} \delta^{-2} \norm[\Lp{2}{\inflateN{\ElementsB}{\delta}}]{u}^2 + p^{\CRed} \delta^{-1} \norm[\Lp{2}{\Omega}]{\kappa\gradN{u}} \norm[\Lp{2}{\inflateN{\ElementsB}{\delta}}]{u}.
\end{equation*}

On the other hand, we can use the definition of $\bilinear[a]{\cdot}{\cdot}$ from \cref{Bilinear_form} to expand the term $\bilinear[a]{u}{\kappa^2 u}$. One of the summands is amenable to the coercivity of the PDE coefficient $a_1$:
\begin{eqnarray*}
\norm[\Lp{2}{\Omega}]{\kappa\gradN{u}}^2 &\cleq& \skalar[\Lp{2}{\Omega}]{a_1\kappa\gradN{u}}{\kappa\gradN{u}} \\
&=& \bilinear[a]{u}{\kappa^2 u} - 2\skalar[\Lp{2}{\Omega}]{a_1\kappa\gradN{u}}{u\gradN{\kappa}} - \skalar[\Lp{2}{\Omega}]{a_2 \* \gradN{u}}{\kappa^2 u} - \skalar[\Lp{2}{\Omega}]{a_3 u}{\kappa^2 u} \\
&\stackrel{\delta \cleq 1 \leq p}{\cleq}& p^{\CRed+2} \delta^{-2} \norm[\Lp{2}{\inflateN{\ElementsB}{\delta}}]{u}^2 + p^{\CRed} \delta^{-1} \norm[\Lp{2}{\Omega}]{\kappa\gradN{u}} \norm[\Lp{2}{\inflateN{\ElementsB}{\delta}}]{u} \\
&\stackrel{\forall \epsilon>0}{\cleq}& C_{\epsilon} p^{2\CRed} \delta^{-2} \norm[\Lp{2}{\inflateN{\ElementsB}{\delta}}]{u}^2 + \epsilon\norm[\Lp{2}{\Omega}]{\kappa\gradN{u}}^2.
\end{eqnarray*}

Since the Young parameter $\epsilon$ can be chosen arbitrarily small, we may absorb the last summand in the left-hand side of the overall inequality. Finally, since $\kappa \equiv 1$ on $\ElementsB$, we obtain the desired Caccioppoli inequality:
\begin{equation*}
\seminorm[\Hk{1}{\ElementsB}]{u} \leq \norm[\Lp{2}{\Omega}]{\kappa\gradN{u}} \cleq p^{\CRed} \delta^{-1} \norm[\Lp{2}{\inflateN{\ElementsB}{\delta}}]{u}.
\end{equation*}

\end{proof}

We close this section with the promised improvement of the discrete Caccioppoli inequality. This time, it will be phrased in terms of the new spaces $\SHarm{B}$ and $\SHarm{\inflateN{B}{\delta}}$, where $B \subseteq \R^d$ is an axes-parallel box, $\delta>0$ is a given parameter, and $\inflateN{B}{\delta} \subseteq \R^d$ is the inflated box in the sense of \cref{Box_infl}. Most importantly, \emph{no} lower bound on $\delta$ is assumed.

The basic idea is to split the elements $T \in \Elements$ touching the inner box $B$ into two groups, based on the relative size of $\h{T}$ and $\delta$. The first group contains the elements that are small in relation to $\delta$ and can therefore be treated with \cref{Space_SHarm_Cacc_old}. The second group contains the larger elements (relative to $\delta$) and we can use an inverse inequality to derive the desired bound. However, since the larger elements might not be fully contained in the outer box $\inflateN{B}{\delta}$, we have to break them up into smaller pieces first.

\begin{lemma} \label{Subdivide_element}
Denote by $\CShape \geq 1$ the shape-regularity constant from \cref{Shape_regularity}. Let $T \in \Elements$ and $\delta>0$ be such that $16\CShape^4 \h{T} > \delta$. Then, there exists a mesh $\ElementsS \subseteq \Pow{T}$ with the following properties:
\begin{enumerate}
\item For all $S,\tilde{S} \in \Elements$ with $S \neq \tilde{S}$, there holds $S \cap \tilde{S} = \emptyset$. Furthermore, $\closureN{\bigcup\ElementsS} = \closureN{T}$.
\item There hold the bounds $\hMax{\ElementsS} \leq \delta \leq C(d,\CShape) \hMin{\ElementsS}$.
\item The mesh $\ElementsS$ is shape-regular in the sense of \cref{Shape_regularity} with a constant $\CShapeTilde = C(d,\CShape)$.
\end{enumerate}
\end{lemma}

\begin{proof}
Denote by $\hat{T} \subseteq \R^d$ the reference simplex, let $M \in \N$ and set $J := \set{1,\dots,M^d}$. In \cite{Edelsbrunner}, it was shown that $\hat{T}$ can be partitioned into $M^d$ simplices $\hat{S}_1,\dots,\hat{S}_{M^d} \subseteq \hat{T}$ of at most $d!/2$ congruence classes, such that $\meas{\hat{S}_j} = M^{-d}\meas{\hat{T}}$. Since the number of congruence classes is uniformly bounded (independent of $M$), one can then show that $C^{-1} \leq M\h{\hat{S}_j} \leq C$, for some constant $C = C(d) \geq 1$.

Now, denote by $\fDef{F_T}{\hat{T}}{T}$ the affine diffeomorphism from \cref{Mesh}. Without proof, we mention that $\norm[2]{\gradN{F_T}} \leq \hat{r}^{-1}\h{T}$, where $\hat{r}>0$ is the radius of the largest ball that can be inscribed into $\hat{T}$. Similarly, exploiting the shape regularity of the mesh $\Elements$ (cf. \cref{Shape_regularity}), there holds $\norm[2]{\grad{F_T^{-1}}} \leq (\CShape^{-1}\h{T})^{-1}\h{\hat{T}} \cleq \h{T}^{-1}$. Then, using the ceiling function $\lceil \cdot \rceil$, we choose
\begin{equation*}
M := \lceil C\hat{r}^{-1}\h{T}\delta^{-1} \rceil \in \N
\end{equation*}
and argue that the system $\ElementsS := \Set{F_T(\hat{S}_j)}{j \in J}$ has the desired properties: Item $(1)$ follows from the fact that the simplices $\hat{S}_j$ partition $\hat{T}$ and item $(3)$ follows from the uniform bound on the number of congruence classes. Finally, to see item $(2)$, we compute
\begin{equation*}
\hMax{\ElementsS} = \max_{j \in J} \h{F_T(\hat{S}_j)} \leq \max_{j \in J} \norm[2]{\gradN{F_T}} \h{\hat{S}_j} \leq C\hat{r}^{-1} \h{T} M^{-1} \stackrel{\text{Def.}M}{\leq} \delta.
\end{equation*}

An analogous computation involving the inverse mapping $\fDef{F_T^{-1}}{T}{\hat{T}}$ reveals the bound $\min_{j \in J} \h{\hat{S}_j} \cleq \h{T}^{-1}\hMin{\ElementsS}$. Furthermore, we invoke the assumption $16\CShape^4 \h{T} > \delta$ to conclude that $M \leq C\hat{r}^{-1}\h{T}\delta^{-1} + 1 \cleq \h{T}\delta^{-1}$. Combining both, we end up with the lower bound
\begin{equation*}
\h{T}^{-1}\delta \cleq C^{-1} M^{-1} \leq \min_{j \in J} \h{\hat{S}_j} \cleq \h{T}^{-1}\hMin{\ElementsS},
\end{equation*}
which readily yields $\delta \cleq \hMin{\ElementsS}$. This concludes the proof.

\end{proof}

Now that we know how to break up an element $T \in \Elements$ into smaller pieces, we present the updated proof of the discrete Caccioppoli inequality.

\begin{lemma} \label{Space_SHarm_Cacc}
Let $B \subseteq \R^d$ be a box and $\delta>0$ with $\delta \cleq 1$. Denote by $\CRed \geq 2$ the constant from \cref{Melenk_Rojik}. Then, for all $u \in \SHarm{\inflateN{B}{\delta}}$, there holds the \emph{Caccioppoli inequality}
\begin{equation*}
\delta \seminorm[\Hk{1}{\Omega \cap B}]{u} \cleq p^{\CRed} \norm[\Lp{2}{\Omega \cap \inflateN{B}{\delta}}]{u}.
\end{equation*}
\end{lemma}

\begin{proof}
First, we apply \cref{Space_SHarm_Cacc_old} to the collection $\ElementsB := \Set{T \in \Elements}{\closureN{T} \cap \closureN{B} \neq \emptyset, 16\CShape^4\h{T} \leq \delta}$ and the parameter $\epsilon := \delta/(4\CShape) > 0$. It is not difficult to see that $4\CShape^3\hMax{\ElementsB} \leq \epsilon \cleq 1$, meaning that $\epsilon$ is indeed a valid parameter choice. Next, let us demonstrate that $\bigcup\inflateN{\ElementsB}{\epsilon} \subseteq \Omega \cap \inflateN{B}{\delta}$: Given $T \in \inflateN{\ElementsB}{\epsilon}$, we know that there exists an element $S \in \ElementsB$ such that $\norm[2]{x_T-x_S} \leq \epsilon$. Since $S$ touches $B$, we can use two triangle inequalities to derive the inclusion $T \subseteq \inflateN{B}{\hMax{\ElementsB} + \epsilon + \hMax{\inflateN{\ElementsB}{\epsilon}}}$. From \cite[L.3.15]{Angleitner_H_matrices_FEM}, we know that $\hMax{\inflateN{\ElementsB}{\epsilon}} \leq \hMax{\ElementsB} + \CShape\epsilon$, which implies $\hMax{\ElementsB} + \epsilon + \hMax{\inflateN{\ElementsB}{\epsilon}} \leq 2\CShape(\hMax{\ElementsB}+\epsilon) \leq \delta$ and ultimately $T \subseteq \inflateN{B}{\delta}$. Since $T \in \inflateN{\ElementsB}{\epsilon}$ was arbitrary, we conclude that indeed $\bigcup\inflateN{\ElementsB}{\epsilon} \subseteq \Omega \cap \inflateN{B}{\delta}$. Now, for every $u \in \SHarm{\inflateN{B}{\delta}}$, we know from \cref{Space_SHarm} that there exists a function $\tilde{u} \in \SkpO{1}{p}{\Elements}$ such that $\restrictN{u}{\Omega \cap \inflateN{B}{\delta}} = \restrictN{\tilde{u}}{\Omega \cap \inflateN{B}{\delta}}$. Additionally, for all $v \in \SkpO{1}{p}{\Elements}$ with $\supp{v} \subseteq \inflateN{B}{\delta}$, we know that $\bilinear[a]{u}{v} = 0$. In particular, $\bilinear[a]{\tilde{u}}{v} = \bilinear[a]{u}{v} = 0$ as well, because $v$ restricts the effective integration domain to $\Omega \cap \inflateN{B}{\delta}$, where $u$ and $\tilde{u}$ coincide. In other words, we are allowed to apply \cref{Space_SHarm_Cacc_old} to the function $\tilde{u}$:
\begin{equation*}
\delta \seminorm[\Hk{1}{\ElementsB}]{u} \ceq \epsilon \seminorm[\Hk{1}{\ElementsB}]{\tilde{u}} \cleq p^{\CRed} \norm[\Lp{2}{\inflateN{\ElementsB}{\epsilon}}]{\tilde{u}} \leq p^{\CRed} \norm[\Lp{2}{\Omega \cap \inflateN{B}{\delta}}]{\tilde{u}} = p^{\CRed} \norm[\Lp{2}{\Omega \cap \inflateN{B}{\delta}}]{u}.
\end{equation*}

Second, consider an element $T \in \Elements$ with $\closureN{T} \cap \closureN{B} \neq \emptyset$ and $16\CShape^4\h{T} > \delta$. Using \cref{Subdivide_element}, we can find a uniform mesh $\ElementsS$ such that $\closureN{\bigcup\ElementsS} = \closureN{T}$ and $\hMax{\ElementsS} \leq \delta \cleq \hMin{\ElementsS}$. Now consider the elements $\ElementsS_B := \Set{S \in \ElementsS}{\closureN{S} \cap \closureN{B} \neq \emptyset}$. Exploiting $\closureN{\bigcup\ElementsS} = \closureN{T}$, it is not difficult to show that $T \cap B \subseteq \closureN{\bigcup\ElementsS_B}$. Furthermore, since $\hMax{\ElementsS} \leq \delta$, an elementary geometric argument proves that $\bigcup\ElementsS_B \subseteq T \cap \inflateN{B}{\delta}$. Now, for every $u \in \SHarm{\inflateN{B}{\delta}}$, we know from \cref{Space_SHarm} that there exists a function $\tilde{u} \in \SkpO{1}{p}{\Elements}$ such that $\restrictN{u}{\Omega \cap \inflateN{B}{\delta}} = \restrictN{\tilde{u}}{\Omega \cap \inflateN{B}{\delta}}$. Then, using the well-known (e.g., \cite{Ditzian_Inverse_inequality}) inverse inequality $\h{S}\seminorm[\Hk{1}{S}]{\tilde{u}} \cleq p^2 \norm[\Lp{2}{S}]{\tilde{u}}$, $S \in \ElementsS_B$, we get
\begin{equation*}
\delta^2 \seminorm[\Hk{1}{T \cap B}]{u}^2 \cleq \hMin{\ElementsS}^2 \seminorm[\Hk{1}{\bigcup\ElementsS_B}]{\tilde{u}}^2 \leq \sum_{S \in \ElementsS_B} \h{S}^2 \seminorm[\Hk{1}{S}]{\tilde{u}}^2 \cleq p^4 \sum_{S \in \ElementsS_B} \norm[\Lp{2}{S}]{\tilde{u}}^2 = p^4 \norm[\Lp{2}{\bigcup\ElementsS_B}]{\tilde{u}}^2 \leq p^4 \norm[\Lp{2}{T \cap \inflateN{B}{\delta}}]{u}^2.
\end{equation*}
Note that the implicit constant from the inverse inequality only depends on the shape regularity constant $\CShapeTilde = C(d,\CShape)$ from \cref{Subdivide_element}.

Finally, for every $u \in \SHarm{\inflateN{B}{\delta}}$, we put the estimates for both groups of elements together:
\begin{equation*}
\delta^2 \seminorm[\Hk{1}{\Omega \cap B}]{u}^2 = \delta^2 \sum_{\substack{T \in \Elements: \\ \closureN{T} \cap \closureN{B} \neq \emptyset}} \seminorm[\Hk{1}{T \cap B}]{u}^2 \leq \delta^2 \seminorm[\Hk{1}{\ElementsB}]{u}^2 + \sum_{\substack{T \in \Elements: \\ \closureN{T} \cap \closureN{B} \neq \emptyset, \\ 16\CShape^4\h{T} > \delta}} \delta^2 \seminorm[\Hk{1}{T \cap B}]{u}^2 \cleq (p^{2\CRed} + p^4) \norm[\Lp{2}{\Omega \cap \inflateN{B}{\delta}}]{u}^2.
\end{equation*}
Noting $\CRed \geq 2$, this finishes the proof.

\end{proof}

\subsection{The low-rank approximation operator} \label{SSec:Low_rank_approx_op}

In our previous construction of the single-step coarsening operator $\fDef{\CoarseningOp{\ElementsB}{\delta}}{\SHarm{\inflateN{\ElementsB}{\delta}}}{\SHarm{\ElementsB}}$, \cite[T.3.31]{Angleitner_H_matrices_FEM}, we used the orthogonal projection $\fDef{\Pi_{\ElementsS}^p}{\Lp{2}{\Omega}}{\Skp{0}{p}{\ElementsS}}$ on a uniform mesh $\ElementsS$ to reduce the overall rank. The output was subsequently fed into the orthogonal projection $\fDef{P_{\ElementsB}}{\Lp{2}{\Omega}}{\SHarm{\ElementsB}}$ in order to generate an element of $\SHarm{\ElementsB}$ again. The existence of $P_{\ElementsB}$ hinged on the fact that $\SHarm{\ElementsB}$ was finite-dimensional and thus a closed subspace of $\Lp{2}{\Omega}$. However, according to \cref{Space_SHarm_Props}, the updated spaces $\SHarm{B}$ from \cref{Space_SHarm} are closed subspaces of $\Hk{1}{\Omega}$, rather than $\Lp{2}{\Omega}$. Therefore, we now have to use the orthogonal projection $\fDef{P_B}{\Hk{1}{\Omega}}{\SHarm{B}}$, and a replacement $\fDef{\Pi_H}{\Hk{1}{\Omega}}{\Hk{1}{\Omega}}$ for the orthogonal projection $\fDef{\Pi_{\ElementsS}^p}{\Lp{2}{\Omega}}{\Skp{0}{p}{\ElementsS}}$ is needed.

\begin{lemma} \label{Low_rank_approx_op_2}
Let $H>0$ be a free parameter. Then, there exists a \emph{low-rank approximation operator}
\begin{equation*}
\fDef{\Pi_H}{\Hk{1}{\Omega}}{\Hk{1}{\Omega}}
\end{equation*}
with the following properties:
\begin{enumerate}
\item \emph{Local rank:} For all boxes $B \subseteq \R^d$, there holds the dimension bound
\begin{equation*}
\dimN{\Set{\Pi_H v}{v \in \Hk{1}{\Omega} \,\, \text{with} \,\, \supp{v} \subseteq B}} \cleq (1+\diam[2]{B}/H)^d.
\end{equation*}

\item \emph{Error bound:} For all $v \in \Hk{1}{\Omega}$, there holds the global error bound
\begin{equation*}
\sum_{l=0}^{1} H^l \seminorm[\Hk{l}{\Omega}]{v-\Pi_H v} \cleq H \seminorm[\Hk{1}{\Omega}]{v}.
\end{equation*}
\end{enumerate}
\end{lemma}

\begin{proof}
Using successive refinements of an arbitrary initial mesh, we can construct a uniform mesh $\ElementsS \subseteq \Pow{\Omega}$ with $\hMax{\ElementsS} \leq H \cleq \hMin{\ElementsS}$. Denote by $\Nodes$ the set of nodes and by $\Set{g_N}{N \in \Nodes} \subseteq \Skp{1}{1}{\ElementsS}$ the corresponding basis of hat functions. We choose the classical \emph{Cl\'ement operator} $\fDef{\Pi_H}{\Lp{2}{\Omega}}{\Skp{1}{1}{\ElementsS}}$ from \cite{clement75}, which maps any given input $v \in \Lp{2}{\Omega}$ to the linear combination $\Pi_H v := \sum_{N \in \Nodes} v_N g_N$, where $v_N \in \R$ is the mean value of $v$ over the support of $g_N$. While the error bound is common knowledge (e.g., \cite[T.1]{clement75}), the dimension bound amounts to counting the number of mesh elements lying inside the slightly inflated box $\inflateN{B}{H}$:
\begin{equation*}
\cardN{\Set{S \in \ElementsS}{S \subseteq \inflateN{B}{H}}} \leq H^{-d} \sum_{S \subseteq \inflateN{B}{H}} \h{S}^d \cleq H^{-d} \sum_{S \subseteq \inflateN{B}{H}} \meas{S} \leq H^{-d} \meas{\inflateN{B}{H}} \cleq (1+\diam[2]{B}/H)^d.
\end{equation*}
\end{proof}

\subsection{The coarsening operators}

At this point, we present an updated construction of the single-step coarsening operator from \cite[T.3.31]{Angleitner_H_matrices_FEM}. The proof is less obfuscated than before, because the tedious case analysis for the parameter $\delta$ has become obsolete.

\begin{theorem} \label{Coarse_op_single}
Let $B \subseteq \R^d$ be a box and $\delta>0$ be a free parameter with $\delta \cleq 1$. Denote by $\CRed \geq 2$ the constant from \cref{Melenk_Rojik}. Then, there exists a linear \emph{single-step coarsening operator}
\begin{equation*}
\fDef{\CoarseningOp{B}{\delta}}{\SHarm{\inflateN{B}{\delta}}}{\SHarm{B}}
\end{equation*}
with the following properties:
\begin{enumerate}
\item \emph{Rank bound:} The rank is bounded by
\begin{equation*}
\rank{\CoarseningOp{B}{\delta}} \cleq p^{d\CRed} (1+\diam[2]{B}/\delta)^d.
\end{equation*}

\item \emph{Approximation error:} For all $u \in \SHarm{\inflateN{B}{\delta}}$, there holds the error bound
\begin{equation*}
\norm[\Lp{2}{\Omega \cap B}]{u - \CoarseningOp{B}{\delta} u} \leq \frac{1}{2} \norm[\Lp{2}{\Omega \cap \inflateN{B}{\delta}}]{u}.
\end{equation*}
\end{enumerate}
\end{theorem}

\begin{proof}
Denote by $\fDef{\CutoffOp{B}{\delta/2}}{\Hk{1}{\Omega}}{\Hk{1}{\Omega}}$ the cut-off operator from \cref{Cut_off_op}. Next, let $H>0$ and denote by $\fDef{\Pi_H}{\Hk{1}{\Omega}}{\Hk{1}{\Omega}}$ the low-rank approximation operator from \cref{Low_rank_approx_op_2}. Furthermore, since $\SHarm{B} \subseteq \Hk{1}{\Omega}$ is a closed subspace (cf. \cref{Space_SHarm_Props}), we may introduce the orthogonal projection $\fDef{P_B}{\Hk{1}{\Omega}}{\SHarm{B}}$ with respect to the equivalent norm $\sum_{l=0}^{1} H^l \seminorm[\Hk{l}{\Omega}]{\cdot}$. We then define the combined operator
\begin{equation*}
\fDef{\CoarseningOp{B}{\delta} := P_B \Pi_H \CutoffOp{B}{\delta/2}}{\SHarm{\inflateN{B}{\delta}}}{\SHarm{B}}.
\end{equation*}

First, we establish the error bound: Let $u \in \SHarm{\inflateN{B}{\delta}}$. From \cref{Space_SHarm_Props} we know that $u \in \SHarm{B}$ and that $\CutoffOp{B}{\delta/2} u \in \SHarm{B}$. Since $P_B$ is a projection onto $\SHarm{B}$, it follows that $P_B \CutoffOp{B}{\delta/2} u = \CutoffOp{B}{\delta/2} u$. Then, using \cref{Cut_off_op_Props}, we get the identity $\restrictN{u}{\Omega \cap B} = \restrict{\CutoffOp{B}{\delta/2} u}{\Omega \cap B} = \restrict{P_B \CutoffOp{B}{\delta/2} u}{\Omega \cap B}$. We compute
\begin{eqnarray*}
\sum_{l=0}^{1} H^l \seminorm[\Hk{l}{\Omega \cap B}]{u-\CoarseningOp{B}{\delta} u} = \sum_{l=0}^{1} H^l \seminorm[\Hk{l}{\Omega \cap B}]{P_B \CutoffOp{B}{\delta/2} u - P_B \Pi_H \CutoffOp{B}{\delta/2} u} = \sum_{l=0}^{1} H^l \seminorm[\Hk{l}{\Omega}]{P_B(\identity-\Pi_H)(\CutoffOp{B}{\delta/2} u)} \\
\leq  \sum_{l=0}^{1} H^l \seminorm[\Hk{l}{\Omega}]{(\identity-\Pi_H)(\CutoffOp{B}{\delta/2} u)} \stackrel{\cref{Low_rank_approx_op_2}}{\cleq} H \seminorm[\Hk{1}{\Omega}]{\CutoffOp{B}{\delta/2} u} \stackrel{\cref{Cut_off_op_Props}}{\cleq} (H/\delta) \sum_{l=0}^{1} \delta^l \seminorm[\Hk{l}{\Omega \cap \inflateN{B}{\delta/2}}]{u} \stackrel{\cref{Space_SHarm_Cacc}}{\cleq} p^{\CRed} (H/\delta) \norm[\Lp{2}{\Omega \cap \inflateN{B}{\delta}}]{u}.
\end{eqnarray*}
Now, denote the implicit cumulative constant by $C>0$. Then, by choosing $H := \delta/(2Cp^{\CRed})>0$, we get the desired factor $1/2$.

Finally, the rank bound can be seen as follows:
\begin{equation*}
\begin{array}{rclcl}
\rank{\CoarseningOp{B}{\delta}} &=& \dimN{\Set{P_B \Pi_H \CutoffOp{B}{\delta/2} u}{u \in \SHarm{\inflateN{B}{\delta}}}} &\stackrel{\cref{Cut_off_op_Props}}{\leq}& \dimN{\Set{\Pi_H v}{v \in \Hk{1}{\Omega} \,\, \text{with} \,\, \supp{v} \subseteq \inflateN{B}{\delta/2}}} \\
&\stackrel{\cref{Low_rank_approx_op_2}}{\cleq}& (1+\diam[2]{\inflateN{B}{\delta/2}}/H)^d &\stackrel{\delta \ceq Hp^{\CRed}}{\cleq}& p^{d\CRed} (1+\diam[2]{B}/\delta)^d.
\end{array}
\end{equation*}
This concludes the proof.

\end{proof}

Now that the new version of the \emph{single}-step coarsening operator $\CoarseningOp{B}{\delta}$ is established, the \emph{multi}-step coarsening operator $\CoarseningOp{B}{\delta,L}$ can be constructed as before:

\begin{theorem} \label{Coarse_op_multi}
Let $B \subseteq \R^d$ be a box and $\delta>0$ be a free parameter with $\delta \cleq 1$. Furthermore, let $L \in \N$. Denote by $\CRed \geq 2$ the constant from \cref{Melenk_Rojik}. Then, there exists a linear \emph{multi-step coarsening operator}
\begin{equation*}
\fDef{\CoarseningOp{B}{\delta,L}}{\SHarm{\inflateN{B}{\delta L}}}{\SHarm{B}}
\end{equation*}
with the following properties:
\begin{enumerate}
\item \emph{Rank bound:} The rank is bounded by
\begin{equation*}
\rank{\CoarseningOp{B}{\delta,L}} \cleq p^{d\CRed} (L+\diam[2]{B}/\delta)^{d+1}.
\end{equation*}

\item \emph{Approximation error:} For all $u \in \SHarm{\inflateN{B}{\delta L}}$, there holds the error bound
\begin{equation*}
\norm[\Lp{2}{\Omega \cap B}]{u - \CoarseningOp{B}{\delta,L} u} \leq 2^{-L} \norm[\Lp{2}{\Omega \cap \inflateN{B}{\delta L}}]{u}.
\end{equation*}
\end{enumerate}
\end{theorem}

\begin{proof}
The basic idea is to combine the single-step coarsening operators $\fDef{\CoarseningOp{B_l}{\delta}}{\SHarm{B_{l+1}}}{\SHarm{B_l}}$ that are associated with the concentric boxes $B_l := \inflateN{B}{\delta l}$, $l \in \set{0,\dots,L}$. The details of the construction can be found in \cite[T.3.32]{Angleitner_H_matrices_FEM}.
\end{proof}

\subsection{Putting everything together}

In this section, we mimick \cite[Section 3.9]{Angleitner_H_matrices_FEM}, and may finally prove our main result, \cref{Main_result}. Recall from \cref{SSec_Reduction_mat_lvl_fct_lvl} that we need to approximate the admissible blocks $\restrictN{\mvemph{A}^{-1}}{I \times J}$ by low-rank matrices in order to get an $\mathcal{H}$-matrix approximation to the full matrix $\mvemph{A}^{-1}$. Then, \cref{Mat_lvl_to_fct_lvl} translated the problem into the realm of function spaces, implying that a suitable subspace $V \subseteq \Lp{2}{\Omega}$ needs to be constructed. We already know from our previous work, \cite[T.3.33]{Angleitner_H_matrices_FEM}, that the range of the multi-step coarsening operator $\CoarseningOp{B}{\delta,L}$ does the trick:

\begin{theorem} \label{Space_VBDL}
Let $B,D \subseteq \R^d$ be two boxes with $0 < \diam[2]{B} \leq \CAdm \dist[2]{B}{D}$. Furthermore, let $L \in \N$. Denote by $\CRed \geq 2$ the constant from \cref{Melenk_Rojik}. Then, there exists a subspace
\begin{equation*}
V_{B,D,L} \subseteq \Lp{2}{\Omega}
\end{equation*}
with the following properties:
\begin{enumerate}
\item \emph{Dimension bound:} There holds the dimension bound
\begin{equation*}
\dimN{V_{B,D,L}} \cleq p^{d\CRed} L^{d+1}.
\end{equation*}

\item \emph{Approximation property:} For every $f \in \Lp{2}{\Omega}$ with $\supp{f} \subseteq D$, there holds the error bound
\begin{equation*}
\inf_{v \in V_{B,D,L}} \norm[\Lp{2}{\Omega \cap B}]{S_{\Elements} f - v} \cleq 2^{-L} \norm[\Lp{2}{\Omega}]{f}.
\end{equation*}
\end{enumerate}
\end{theorem}

\begin{proof}
Let $B,D \subseteq \R^d$ and $L \in \N$ as above. Set $\delta := \diam[2]{B}/(2\sqrt{d}\CAdm L) > 0$ and denote by $\fDef{\CoarseningOp{B}{\delta,L}}{\SHarm{\inflateN{B}{\delta L}}}{\SHarm{B}}$ the multi-step coarsening operator from \cref{Coarse_op_multi}. We choose the space
\begin{equation*}
V_{B,D,L} := \ran{\CoarseningOp{B}{\delta,L}} \subseteq \SHarm{B} \subseteq \Lp{2}{\Omega}.
\end{equation*}

Using \cref{Coarse_op_multi} and the definition of $\delta$, we can bound the dimension as follows:
\begin{equation*}
\dimN{V_{B,D,L}} = \rank{\CoarseningOp{B}{\delta,L}} \cleq p^{d\CRed} (L+\diam[2]{B}/\delta)^{d+1} \cleq p^{d\CRed} L^{d+1}.
\end{equation*}

Finally, let $f \in \Lp{2}{\Omega}$ with $\supp{f} \subseteq D$. In order to show that the error bound from \cref{Coarse_op_multi} is applicable to the function $S_{\Elements} f \in \SkpO{1}{p}{\Elements}$, we first need to establish the fact that $S_{\Elements} f \in \SHarm{\inflateN{B}{\delta L}}$. According to \cref{Space_SHarm_Props}, it suffices to prove that the sets $\inflateN{B}{\delta L}$ and $D$ are disjoint. To that end, we choose a point $z \in \closureN{\inflateN{B}{\delta L}}$ with $\dist[2]{\inflateN{B}{\delta L}}{D} = \dist[2]{z}{D}$. Then, $\dist[2]{B}{D} \leq \dist[2]{B}{z} + \dist[2]{z}{D} \leq \sqrt{d} \delta L + \dist[2]{\inflateN{B}{\delta L}}{D}$. Combined with the definition of $\delta$ and the admissibility condition, this yields
\begin{equation*}
\dist[2]{\inflateN{B}{\delta L}}{D} \geq \dist[2]{B}{D} - \sqrt{d} \delta L = \dist[2]{B}{D} - \diam[2]{B}/(2\CAdm) \geq \diam[2]{B}/(2\CAdm) > 0.
\end{equation*}

\cref{Space_SHarm_Props} implies $S_{\Elements} f \in \SHarm{\inflateN{B}{\delta L}}$, so that $\CoarseningOp{B}{\delta,L}(S_{\Elements} f) \in V_{B,D,L}$. Hence, the error bound from \cref{Coarse_op_multi} is applicable to the function $S_{\Elements} f$. Using the {\sl a priori} stability bound of the discrete solution operator $\SolutionOp{\Elements}$ (cf. \cref{Sol_op_disc}), we then estimate
\begin{equation*}
\inf_{v \in V_{B,D,L}} \norm[\Lp{2}{\Omega \cap B}]{S_{\Elements} f - v} \leq \norm[\Lp{2}{\Omega \cap B}]{S_{\Elements} f - \CoarseningOp{B}{\delta,L}(S_{\Elements} f)} \leq 2^{-L} \norm[\Lp{2}{\Omega \cap \inflateN{B}{\delta L}}]{S_{\Elements} f} \cleq 2^{-L} \norm[\Lp{2}{\Omega}]{f}.
\end{equation*}
This concludes the proof.
\end{proof}

Finally, we have everything we need to derive our main result:

\begin{proof}[Proof of \cref{Main_result}]

Let $\mvemph{A} \in \R^{N \times N}$ be the system matrix from \cref{System_matrix} and $r \in \N$ a given block rank bound. We define the asserted $\mathcal{H}$-matrix approximant $\mvemph{B} \in \R^{N \times N}$ in a block-wise fashion:

First, consider an admissible block $(I,J) \in \BPartAdm$. From \cref{Block_partition} we know that there exist boxes $B_I,B_J \subseteq \R^d$ with $\Omega_I \subseteq B_I$, $\Omega_J \subseteq B_J$ and $\diam[2]{B_I} \leq \CAdm \dist[2]{B_I}{B_J}$. In particular, $\diam[2]{B_I} \geq \diam[2]{\Omega_I} > 0$, so that \cref{Space_VBDL} is applicable to $B_I$ and $B_J$. Now, denote by $C>0$ the implicit constant from the dimension bound in \cref{Space_VBDL}. We set $\CExp := \ln(2)/C^{1/(d+1)} > 0$ and $L := \lfloor (r/C)^{1/(d+1)} p^{-\CRed} \rfloor \in \N$. Then, \cref{Space_VBDL} provides a subspace $V_{I,J,r} \subseteq V$. We apply \cref{Mat_lvl_to_fct_lvl} to this subspace and get an integer $\tilde{r} \leq \dimN{V_{I,J,r}}$ and matrices $\mvemph{X}_{I,J,r} \in \R^{I \times \tilde{r}}$ and $\mvemph{Y}_{I,J,r} \in \R^{J \times \tilde{r}}$. We set
\begin{equation*}
\restrictN{\mvemph{B}}{I \times J} := \mvemph{X}_{I,J,r} (\mvemph{Y}_{I,J,r})^T.
\end{equation*}

Second, for every small block $(I,J) \in \BPartSmall$, we make the trivial choice
\begin{equation*}
\restrictN{\mvemph{B}}{I \times J} := \restrictN{\mvemph{A}^{-1}}{I \times J}.
\end{equation*}

By \cref{H_matrices}, we have $\mvemph{B} \in \HMatrices{\BPart}{\tilde{r}}$ with a block rank bound
\begin{equation*}
\tilde{r} \leq \dimN{V_{I,J,r}} \stackrel{\text{Def.} \, C}{\leq} C p^{d\CRed} L^{d+1} \leq C (p^{\CRed} L)^{d+1} \stackrel{\text{Def.} \, L}{\leq} r.
\end{equation*}

As for the error, we get
\begin{eqnarray*}
\norm[2]{\mvemph{A}^{-1} - \mvemph{B}} &\stackrel{\cref{Block_partition}}{\cleq}& \ln(\hMin{}^{-d}) \max_{(I,J) \in \BPartAdm} \norm[2]{\restrictN{\mvemph{A}^{-1}}{I \times J} - \mvemph{X}_{I,J,r} (\mvemph{Y}_{I,J,r})^T} \\
&\stackrel{\cref{Mat_lvl_to_fct_lvl}}{\cleq}& p^{2\CStab} \ln(\hMin{}^{-d}) \hMin{}^{-d} \max_{(I,J) \in \BPartAdm} \sup_{\substack{f \in V_{I,J,r}: \\ \supp{f} \subseteq \Omega_J}} \inf_{v \in V_{I,J,r}} \frac{\norm[\Lp{2}{\Omega_I}]{S_{\Elements} f - v}}{\norm[\Lp{2}{\Omega}]{f}} \\
&\stackrel{\cref{Space_VBDL}}{\cleq}& p^{2\CStab} \ln(\hMin{}^{-d}) \hMin{}^{-d} 2^{-L} \\
&\stackrel{\text{Def.} \, L}{\cleq}& p^{2\CStab} \ln(\hMin{}^{-d}) \hMin{}^{-d} \exp(-\ln(2)(r/C)^{1/(d+1)} p^{-\CRed}) \\
&\stackrel{\text{Def.} \, \CExp}{=}& p^{2\CStab} \ln(\hMin{}^{-d}) \hMin{}^{-d} \exp(-\CExp r^{1/(d+1)} p^{-\CRed}),
\end{eqnarray*}
which finishes the proof.

\end{proof}

\section{Polynomial preserving lifting from the boundary and an elementwise defined projection} \label{Sec_Rojik_proof}

In this section, we provide the proof of \cref{Melenk_Rojik}, i.e., we devise an approximation operator $\fDef{J_{\Elements}^p}{\Skp{0}{p+2}{\Elements}}{\Skp{0}{p}{\Elements}}$ that is defined in a elementwise fashion and preserves global continuity. In other words, we need to approximate a spline $u \in \SkpO{1}{p+2}{\Elements}$ of degree $p+2$ by a spline $\tilde{u} \in \SkpO{1}{p}{\Elements}$ of degree $p$ in a way that is stable in $p$.


The subsequent construction of $J_{\Elements}^p$ generalizes \cite[L.4.1., D.2.5., D.2.1.]{Melenk_Rojik} from $d \in \set{1,2,3}$ to arbitrary spatial dimension $d \geq 1$. In \cite{Melenk_Rojik}, $J_{\Elements}^p$ was defined in a piecewise manner by means of an operator $\fDef{\hat{J}^p}{\Hk{(d+1)/2}{\hat{T}}}{\Pp{p}{\hat{T}}}$ on the reference simplex $\hat{T} \subseteq \R^d$. While the results from \cite{Melenk_Rojik} produce the optimal powers of $p$, the proofs are rather involved due to the nonlocality of the pertinent fractional Sobolev norms. In the present paper, we only need the specific case of \emph{polynomial} inputs $f \in \Pp{p+2}{\hat{T}}$. Therefore, using inverse inequalities, we may work with the much simpler norms $\norm[\Lp{2}{\hat{T}}]{\cdot}$ and $\norm[\Lp{2}{\boundaryN{\hat{T}}}]{\cdot}$ at the expense of powers of $p$.
The definition of the operator $\hat{J}^p$ from \cite{Melenk_Rojik} can easily be generalized to arbitrary space dimension $d \geq 1$. However, in order to derive error estimates, a polynomial preserving lifting operator has to be used. The literature on polynomial preserving liftings is extensive (e.g., \cite{Katz_Polynomial_lifting,babuska-craig-mandel-pitkaranta91, bernardi-dauge-maday92,Munoz_Polynomial_lifting,bernardi-maday97,bernardi-dauge-maday07}), but many authors focus on the special cases $d \in \set{2,3}$ and stability estimates are usually phrased in terms of the norm $\norm[\Hk{1/2}{\boundaryN{\hat{T}}}]{\cdot}$. In the sequel, we present a polynomial preserving lifting for arbitrary space dimension $d \geq 1$ that seems to have been overlooked in the pertinent literature. As usual, we first devise a lifting from one of $\hat{T}$'s hyperplanes $\hat{\Gamma} \subseteq \boundaryN{\hat{T}}$ into its interior (cf. \cref{Lifting_op_single}). Then, in \cref{Lifting_op}, we combine the liftings of all such $\hat{\Gamma}$.

To get things going, let $d \geq 1$ as before and consider the \emph{reference $d$-simplex} $\hat{T} := \hat{T}^d := \Set{x \in [0,1]^d}{\norm[1]{x} \leq 1} \subseteq \R^d$. (In this section, we use the closed version in order to ease notation.) We denote by $\Nodes(\hat{T}) := \set{0,e_1,\dots,e_d}$ its set of nodes, $0 \in \R^d$ being the origin and $e_i \in \R^d$ being the $i$-th Euclidean unit vector. In order to describe the boundary $\boundaryN{\hat{T}}$ efficiently, let us introduce \emph{$k$-simplices}. The definition uses the notion of convex hulls, $\convexhull{\Omega} := \Set{(1-t)x+ty}{x,y \in \Omega, t \in [0,1]}$ for all $\Omega \subseteq \R^d$.

\begin{definition} \label{Simplex}
Let $k \in \set{0,\dots,d}$. A subset $\hat{\Sigma} \subseteq \hat{T}$ is called \emph{$k$-simplex}, if there exist $k+1$ distinct nodes $\hat{N}_0,\dots,\hat{N}_k \in \Nodes(\hat{T})$ such that $\hat{\Sigma} = \convexhullN{\set{\hat{N}_0,\dots,\hat{N}_k}}$.
\end{definition}

(Again, we think of $k$-simplices $\hat{\Sigma}$ as being closed.) Note that $\hat{\Sigma} \subseteq \boundaryN{\hat{T}}$, if $k \leq d-1$, and $\hat{\Sigma} = \hat{T}$, if $k=d$. Recall that any $k$-simplex $\hat{\Sigma} = \convexhullN{\set{\hat{N}_0,\dots,\hat{N}_k}}$ is isomorphic to the reference $k$-simplex $\hat{T}^k = \Set{t \in [0,1]^k}{\norm[1]{t} \leq 1} \subseteq \R^k$. In fact, consider the affine parametrization $\fDef{\sigma}{\hat{T}^k}{\hat{\Sigma}}$, $\sigma(t) := \hat{N}_0 + \sum_{i=1}^{k} t_i(\hat{N}_i-\hat{N}_0)$. Then, there holds the representation
\begin{equation} \label{Convex_hull_param}
\hat{\Sigma} = \convexhullN{\set{\hat{N}_0,\dots,\hat{N}_k}} = \Set{\sigma(t)}{t \in \hat{T}^k}.
\end{equation}

Next, let us introduce some function spaces.

\begin{definition} \label{Space_Pp_boundary}
Let $k \in \set{0,\dots,d}$ and consider a $k$-simplex $\hat{\Sigma} \subseteq \hat{T}$ along with an affine parametrization $\fDef{\sigma}{\hat{T}^k}{\hat{\Sigma}}$. We define the spaces
\begin{equation*}
\begin{array}{rcl}
\Pp{p}{\hat{\Sigma}} &:=& \Set{\fDef{f}{\hat{\Sigma}}{\R}}{f \circ \sigma \in \Pp{p}{\hat{T}^k}}, \\
\Pp{p}{\boundaryN{\hat{T}}} &:=& \Set{f \in \Ck{0}{\boundaryN{\hat{T}}}}{\forall (d-1)\text{-simplices} \,\, \hat{\Gamma} \subseteq \hat{T}: \restrictN{f}{\hat{\Gamma}} \in \Pp{p}{\hat{\Gamma}}}.
\end{array}
\end{equation*}
\end{definition}

Note that a function $f \in \Pp{p}{\boundaryN{\hat{T}}}$ necessarily satisfies some compatibility conditions along all $(d-2)$-simplices $\hat{\Sigma} \subseteq \hat{T}$.

Before we construct the lifting operator in an arbitrary space dimension $d \geq 1$, let us first look at the case $d=3$, i.e., $\hat{T} = \Set{(x_1,x_2,x_3)}{x_i \in [0,1], x_1+x_2+x_3 \leq 1}$. We enumerate the nodes as $\hat{N}_0 := (0,0,0)$, $\hat{N}_1 := (1,0,0)$, $\hat{N}_2 := (0,1,0)$ and $\hat{N}_3 := (0,0,1)$. Now, looking at \cref{Figure_Lifting_operator}, our goal is to find a lifting from the bottom face $\hat{\Gamma} := \convexhullN{\set{\hat{N}_0, \hat{N}_1, \hat{N}_2}} = \Set{x \in \hat{T}}{x_3 = 0}$ upwards, into the $x_3$-dimension.

\begin{figure}[H]
\begin{center}
\includegraphics[width=0.5\textwidth, trim=0cm 0cm 0cm 0cm]{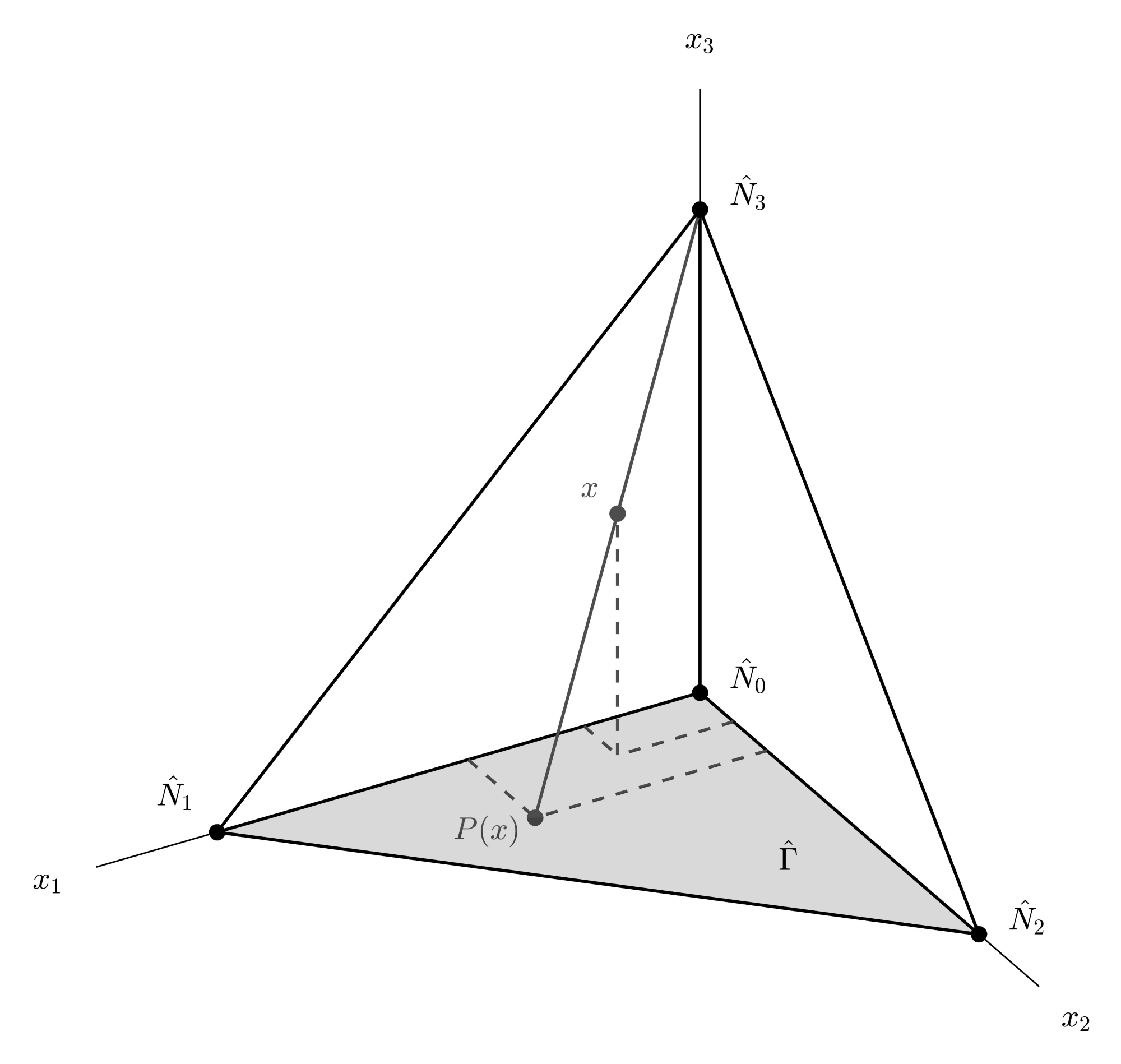}
\caption{The lifting operator in the case $d=3$.}
\label{Figure_Lifting_operator}
\end{center}
\end{figure}

Consider given boundary data $f \in \Ck{0}{\hat{\Gamma}}$. Given any point $y \in \hat{\Gamma}$, the basic idea is to propagate the value $f(y)$ along the line segment from $y$ to $\hat{N}_3$. To be more precise, let us denote the result of the lifting process by $\hat{L}_{\hat{\Gamma}} f \in \Ck{0}{\hat{T}}$. In order to define the value $(\hat{L}_{\hat{\Gamma}} f)(x)$, for any given point $x \in \hat{T}\backslash\set{\hat{N}_3}$, we proceed as follows: First, we cast a ray from the top node $\hat{N}_3$ through the given point $x$ and compute the intersection point with the bottom face $\hat{\Gamma}$. In fact, this intersection point is given by $P(x) := (x_1/(1-x_3), x_2/(1-x_3), 0)$. Then, we set
\begin{equation*}
(\hat{L}_{\hat{\Gamma}} f)(x) := (1-x_3)^p f(P(x)) = (1-x_3)^p f\(\frac{x_1}{1-x_3}, \frac{x_2}{1-x_3}, 0\).
\end{equation*}

The purpose of the prefactor $(1-x_3)^p$ is to guarantee that $\hat{L}_{\hat{\Gamma}} f \in \Pp{p}{\hat{T}}$, whenever $f \in \Pp{p}{\hat{\Gamma}}$, by ``undoing'' the division by $(1-x_3)$ inside the argument of $f$. In fact, if we plug in $f(x) = \sum_{\abs{\alpha} \leq p} f_{\alpha}x^{\alpha}$, we can see that
\begin{equation*}
(\hat{L}_{\hat{\Gamma}} f)(x) = (1-x_3)^p f\(\frac{x_1}{1-x_3}, \frac{x_2}{1-x_3}, 0\) = \sum_{\abs{\alpha} \leq p} f_{\alpha} (1-x_3)^{p-\abs{\alpha}} x_1^{\alpha_1} x_2^{\alpha_2} 0^{\alpha_3} \in \Pp{p}{\hat{T}}.
\end{equation*}

Note that, since $f$ is bounded, we get the added benefit of $\lim_{x \rightarrow \hat{N}_3} (\hat{L}_{\hat{\Gamma}} f)(x) = 0$. Finally, the prefactor satisfies $(1-x_3)^p = 1$, for all $x \in \hat{\Gamma}$, so that $\restrict{\hat{L}_{\hat{\Gamma}} f}{\hat{\Gamma}} = f$.

Now let us have a look at what happens at the edges and faces connecting $\hat{\Gamma}$ with $\hat{N}_3$. Assume, for example, that the input $f \in \Ck{0}{\hat{\Gamma}}$ vanishes at one of $\hat{\Gamma}$'s nodes, say, $f(\hat{N}_0) = 0$. Then it is immediately clear from the picture in \cref{Figure_Lifting_operator} that $(\hat{L}_{\hat{\Gamma}} f)(x) = 0$ for all $x$ on the edge $\convexhullN{\set{\hat{N}_0,\hat{N}_3}}$. (More rigorously, if $x = (0,0,x_3)$, then $(\hat{L}_{\hat{\Gamma}} f)(x) = (1-x_3)^p f(0,0,0) = 0$.) As another example, consider the case of an input $g \in \Ck{0}{\hat{\Gamma}}$ that vanishes on one of $\hat{\Gamma}$'s edges, say, $g=0$ on $\convexhullN{\set{\hat{N}_0,\hat{N}_1}}$. Then, again by \cref{Figure_Lifting_operator}, we can see that $(\hat{L}_{\hat{\Gamma}} g)(x) = 0$ for all $x$ on the face $\convexhullN{\set{\hat{N}_0,\hat{N}_1,\hat{N}_3}}$. As a mnemonic, we may say that the operator $\hat{L}_{\hat{\Gamma}}$ lifts zeros on $k$-simplices to zeros on $(k+1)$-simplices.

This concludes our introductory example in $d=3$ space dimensions and we are now ready to treat the general case $d \geq 1$.

\begin{lemma} \label{Lifting_op_single}
Let $\hat{\Gamma} \subseteq \hat{T}$ be a $(d-1)$-simplex, say, $\hat{\Gamma} = \convexhullN{\set{\hat{N}_0,\dots,\hat{N}_{d-1}}}$. Denote the remaining node of $\hat{T}$ by $\hat{N}_d \in \Nodes(\hat{T})$. Then, there exists a lifting operator $\fDef{\hat{L}_{\hat{\Gamma}}}{\Ck{0}{\hat{\Gamma}}}{\Ck{0}{\hat{T}}}$ with the following properties:
\begin{enumerate}
\item For every $f \in \Ck{0}{\hat{\Gamma}}$, there holds $\restrict{\hat{L}_{\hat{\Gamma}} f}{\hat{\Gamma}} = f$.

\item For every $f \in \Pp{p}{\hat{\Gamma}}$, there holds $\hat{L}_{\hat{\Gamma}} f \in \Pp{p}{\hat{T}}$.

\item For every $f \in \Ck{0}{\hat{\Gamma}}$, there holds $(\hat{L}_{\hat{\Gamma}} f)(\hat{N}_d) = 0$.

\item Let $k \in \set{0,\dots,d-1}$ and let $\hat{\Sigma} \subseteq \hat{T}$ be a $k$-simplex with $\hat{\Sigma} \subseteq \hat{\Gamma}$. Furthermore, consider the $(k+1)$-simplex $\hat{\Sigma}^+ := \convexhull{\hat{\Sigma} \cup \set{\hat{N}_d}} \subseteq \hat{T}$. Then, for every $f \in \Ck{0}{\hat{\Gamma}}$ with $\restrictN{f}{\hat{\Sigma}} = 0$, there holds $\restrict{\hat{L}_{\hat{\Gamma}} f}{\hat{\Sigma}^+} = 0$.

\item For all $f \in \Ck{0}{\hat{\Gamma}}$, there holds the stability bound
\begin{equation*}
\norm[\Lp{2}{\hat{T}}]{\hat{L}_{\hat{\Gamma}} f} \cleq p^{-1/2} \norm[\Lp{2}{\hat{\Gamma}}]{f}.
\end{equation*}
(In the case $d=1$, we interpret $\norm[\Lp{2}{\hat{\Gamma}}]{f} = \norm[\lp{2}{\hat{\Gamma}}]{f}$.)

\end{enumerate}
\end{lemma}

\begin{proof}
First, since the vectors $\set{\hat{N}_1-\hat{N}_0,\dots,\hat{N}_d-\hat{N}_0} \subseteq \R^d$ form a basis, we may pick a normal vector $n \in \R^d$ of $\hat{\Gamma}$ such that $\skalar{\hat{N}_1-\hat{N}_0}{n} = \dots = \skalar{\hat{N}_{d-1}-\hat{N}_0}{n} = 0$ and $\skalar{\hat{N}_d-\hat{N}_0}{n} = 1$. Note that $n$ can be used to write $\hat{\Gamma}$ in the normal form
\begin{equation}
\label{eq:hatGamma-normal-form}
\hat{\Gamma} = \Set{x \in \hat{T}}{\skalar{x-\hat{N}_0}{n} = 0}.
\end{equation}

For every $x \in \hat{T} \backslash \set{\hat{N}_d}$, the line passing through $\hat{N}_d$ and $x$ is given by $\Set{\hat{N}_d + s(x-\hat{N}_d)}{s \in \R}$. Using the normal form (\ref{eq:hatGamma-normal-form}), the intersection point with $\hat{\Gamma}$ can easily be computed:
\begin{equation*}
P(x) := \hat{N}_d + \skalar{\hat{N}_d-x}{n}^{-1}(x-\hat{N}_d) \in \hat{\Gamma}.
\end{equation*}
Let us verify that indeed $\skalar{\hat{N}_d-x}{n} \neq 0$: Since $x \in \hat{T} \backslash \set{\hat{N}_d}$, we know from \cref{Convex_hull_param} that it can be written in the form $x = \hat{N}_0 + \sum_{i=1}^{d} t_i(\hat{N}_i-\hat{N}_0)$, where $t_d \in [0,1)$. But then $\skalar{\hat{N}_d-x}{n} = \skalar{\hat{N}_d-\hat{N}_0}{n} - \sum_{i=1}^{d} t_i\skalar{\hat{N}_i-\hat{N}_0}{n} = 1-t_d > 0$.

Now, for every $f \in \Ck{0}{\hat{\Gamma}}$, consider the lifting $\hat{L}_{\hat{\Gamma}} f$ defined as follows:
\begin{equation*}
\begin{array}{llcl}
\forall x \in \hat{T}\backslash\set{\hat{N}_d}: \quad \quad & (\hat{L}_{\hat{\Gamma}} f)(x) &:=& \skalar{\hat{N}_d-x}{n}^p f(P(x)), \\
 & (\hat{L}_{\hat{\Gamma}} f)(\hat{N}_d) &:=& 0.
\end{array}
\end{equation*}

Since $f$ is bounded, it is easy to check that $\hat{L}_{\hat{\Gamma}} f \in \Ck{0}{\hat{T}}$. Furthermore, we have $\restrict{\hat{L}_{\hat{\Gamma}} f}{\hat{\Gamma}} = f$, which follows from the identities $\skalar{\hat{N}_d-x}{n} = 1$ and $P(x) = x$, for all $x \in \hat{\Gamma}$.

Next, consider the case $f \in \Pp{p}{\hat{\Gamma}}$. Expanding $f(x) = \sum_{\abs{\alpha} \leq p} f_{\alpha} x^{\alpha}$, we can see that $\hat{L}_{\hat{\Gamma}} f$ is a polynomial as well:
\begin{eqnarray*}
\forall x \in \hat{T}: \quad \quad (\hat{L}_{\hat{\Gamma}} f)(x) &=& \skalar{\hat{N}_d-x}{n}^p \sum_{\abs{\alpha} \leq p} f_{\alpha} (\hat{N}_d + \skalar{\hat{N}_d-x}{n}^{-1}(x-\hat{N}_d))^{\alpha} \\
&=& \sum_{\abs{\alpha} \leq p} f_{\alpha} \skalar{\hat{N}_d-x}{n}^{p-\abs{\alpha}} (\skalar{\hat{N}_d-x}{n}\hat{N}_d + x - \hat{N}_d)^{\alpha} \in \Pp{p}{\hat{T}}.
\end{eqnarray*}

Now, let $k \in \set{0,\dots,d-1}$ and consider a $k$-simplex $\hat{\Sigma} \subseteq \hat{T}$ with $\hat{\Sigma} \subseteq \hat{\Gamma}$. Let $\hat{\Sigma}^+ := \convexhull{\hat{\Sigma} \cup \set{\hat{N}_d}}$ and consider a function $f \in \Ck{0}{\hat{\Gamma}}$ with $\restrictN{f}{\hat{\Sigma}} = 0$. In order to prove the identity $\restrict{\hat{L}_{\hat{\Gamma}} f}{\hat{\Sigma}^+} = 0$, let $x \in \hat{\Sigma}^+$ be given. In the case $x=\hat{N}_d$, we immediately get $(\hat{L}_{\hat{\Gamma}} f)(x) = 0$ from the definition of $\hat{L}_{\hat{\Gamma}} f$. In the non-trivial case $x \neq \hat{N}_d$, we know that there exist $y \in \hat{\Sigma}$ and $t_d \in [0,1]$ such that $x = (1-t_d)y + t_d\hat{N}_d$. Since $y \in \hat{\Sigma} \subseteq \hat{\Gamma}$, we have $P(x) = P(y) = y$, so that $(\hat{L}_{\hat{\Gamma}} f)(x) = \skalar{\hat{N}_d-x}{n}^p f(y) = 0$.

Finally, as for the stability bound, we only prove the non-trivial case $d \geq 2$. To this end, we consider the following parametrizations of $\hat{\Gamma}$ and $\hat{T}$:
\begin{equation*}
\fDefB[\gamma]{\hat{T}^{d-1}}{\hat{\Gamma}}{t}{\hat{N}_0 + \sum_{i=1}^{d-1} t_i(\hat{N}_i-\hat{N}_0)}, \quad \quad \quad \fDefB[\tau]{\hat{T}^{d-1} \times [0,1]}{\hat{T}}{(t,t_d)}{(1-t_d)\gamma(t) + t_d\hat{N}_d}.
\end{equation*}

There hold the identities
\begin{equation*}
\sqrt{\det{(\gradN{\gamma})(t)^T (\gradN{\gamma})(t)}} = C_1, \quad \quad \quad \abs{\det{(\gradN{\tau})(t,t_d)}} = (1-t_d)^{d-1} C_2,
\end{equation*}
where $C_1 := \sqrt{\det{(\skalar{\hat{N}_i-\hat{N}_0}{\hat{N}_j-\hat{N}_0})_{i,j=1}^{d-1}}}$ and $C_2 := \abs{\det{(\hat{N}_1-\hat{N}_0|\dots|\hat{N}_d-\hat{N}_0)}}$. Exploiting the relations $\skalar{\hat{N}_d-\tau(t,t_d)}{n} = 1-t_d$ and $P(\tau(t,t_d)) = P(\gamma(t)) = \gamma(t)$, we compute, for all $f \in \Ck{0}{\hat{\Gamma}}$,
\begin{eqnarray*}
\norm[\Lp{2}{\hat{T}}]{\hat{L}_{\hat{\Gamma}} f}^2 &=& \I{\hat{T}}{\skalar{\hat{N}_d-x}{n}^{2p} f(P(x))^2}{x} \\
&=& \I[1]{0}{\I{\hat{T}^{d-1}}{\skalar{\hat{N}_d-\tau(t,t_d)}{n}^{2p} f(P(\tau(t,t_d)))^2 \abs{\det{(\gradN{\tau})(t,t_d)}}}{t}}{t_d} \\
&=& \frac{C_2}{C_1} \(\I[1]{0}{(1-t_d)^{2p+d-1}}{t_d}\)\(\I{\hat{T}^{d-1}}{f(\gamma(t))^2 \sqrt{\det{(\gradN{\gamma})(t)^T (\gradN{\gamma})(t)}}}{t}\) \\
&=& \frac{C_2}{C_1} (2p+d)^{-1} \norm[\Lp{2}{\hat{\Gamma}}]{f}^2.
\end{eqnarray*}

This finishes the proof.

\end{proof}

Now that the lifting operators for all $(d-1)$-simplices $\hat{\Gamma} \subseteq \boundaryN{\hat{T}}$ are available, we can combine them.

\begin{lemma} \label{Lifting_op}
There exists a \emph{lifting operator} $\fDef{\hat{L}}{\Ck{0}{\boundaryN{\hat{T}}}}{\Ck{0}{\hat{T}}}$ with the following properties:
\begin{enumerate}
\item For every $f \in \Ck{0}{\boundaryN{\hat{T}}}$, there holds $\restrict{\hat{L} f}{\boundaryN{\hat{T}}} = f$.

\item For every $f \in \Pp{p}{\boundaryN{\hat{T}}}$, there holds $\hat{L} f \in \Pp{p}{\hat{T}}$.

\item For all $f \in \Pp{p}{\boundaryN{\hat{T}}}$, there holds the stability estimate
\begin{equation*}
\norm[\Lp{2}{\hat{T}}]{\hat{L} f} \cleq p^{(d-2)/2} \norm[\Lp{2}{\boundaryN{\hat{T}}}]{f}.
\end{equation*}
(In the case $d=1$, we interpret $\norm[\Lp{2}{\boundaryN{\hat{T}}}]{f} = \norm[\lp{2}{\boundaryN{\hat{T}}}]{f}$.)

\end{enumerate}
\end{lemma}

\begin{proof}
For all $(d-1)$-simplices $\hat{\Gamma} \subseteq \hat{T}$, denote by $\fDef{\hat{L}_{\hat{\Gamma}}}{\Ck{0}{\hat{\Gamma}}}{\Ck{0}{\hat{T}}}$ the corresponding lifting operators from \cref{Lifting_op_single}. First, we define an auxiliary operator $\fDef{\hat{M}}{\Ck{0}{\boundaryN{\hat{T}}}}{\Ck{0}{\hat{T}}}$: For every $f \in \Ck{0}{\boundaryN{\hat{T}}}$, we set $\hat{M} f := \sum_{\hat{\Gamma}} \hat{L}_{\hat{\Gamma}} f \in \Ck{0}{\hat{T}}$, where $\hat{L}_{\hat{\Gamma}} f$ is meant as an abbreviation for $\hat{L}_{\hat{\Gamma}}(\restrictN{f}{\hat{\Gamma}})$. Before we construct the alleged operator $\hat{L}$ from $\hat{M}$, let us first present the relevant properties of $\hat{M}$.

Clearly, if $f \in \Pp{p}{\boundaryN{\hat{T}}}$, then $\hat{M} f \in \Pp{p}{\hat{T}}$, by item $(2)$ of \cref{Lifting_op_single}.

Let $f \in \Ck{0}{\boundaryN{\hat{T}}}$ and $\hat{N} \in \Nodes(\hat{T})$ be given. For each $(d-1)$-simplex $\hat{\Gamma} \subseteq \hat{T}$, we distinguish between two cases: First, if $\hat{N} \in \hat{\Gamma}$, then $(\hat{L}_{\hat{\Gamma}} f)(\hat{N}) = f(\hat{N})$ by item $(1)$ of \cref{Lifting_op_single}. Second, if $\hat{N} \notin \hat{\Gamma}$, then item $(3)$ of \cref{Lifting_op_single} immediately tells us that $(\hat{L}_{\hat{\Gamma}} f)(\hat{N}) = 0$. Since the total number of $(d-1)$-simplices $\hat{\Gamma} \subseteq \hat{T}$ is given by $d+1$, and since only one of them falls into the second category, we end up with the following identity:
\begin{equation} \label{Lifting_op_Eq_1}
(\hat{M} f)(\hat{N}) = \sum_{\hat{\Gamma}: \, \hat{N} \in \hat{\Gamma}} (\hat{L}_{\hat{\Gamma}} f)(\hat{N}) + \sum_{\hat{\Gamma}: \, \hat{N} \notin \hat{\Gamma}} (\hat{L}_{\hat{\Gamma}} f)(\hat{N}) = d \cdot f(\hat{N}).
\end{equation}

Next, let $k \in \set{0,\dots,d-1}$ and let $f \in \Ck{0}{\boundaryN{\hat{T}}}$ be such that $\restrictN{f}{\hat{\Sigma}} = 0$, for all $k$-simplices $\hat{\Sigma} \subseteq \hat{T}$. Furthermore, let $\hat{\Sigma}^+ \subseteq \hat{T}$ be an arbitrary $(k+1)$-simplex. Considering a $(d-1)$-simplex $\hat{\Gamma} \subseteq \hat{T}$, we distinguish between two cases again: First, if $\hat{\Sigma}^+ \subseteq \hat{\Gamma}$, then $\restrict{\hat{L}_{\hat{\Gamma}} f}{\hat{\Sigma}^+} = \restrictN{f}{\hat{\Sigma}^+}$ by item $(1)$ of \cref{Lifting_op_single}. Second, if $\hat{\Sigma}^+ \not\subseteq \hat{\Gamma}$, then there must hold $\hat{\Sigma}^+ = \convexhull{\hat{\Sigma} \cup \set{\hat{N}}}$, where $\hat{\Sigma} \subseteq \hat{T}$ is a $k$-simplex with $\hat{\Sigma} \subseteq \hat{\Gamma}$ and where $\hat{N} \in \Nodes(\hat{T}) \backslash \hat{\Gamma}$. Since $\restrictN{f}{\hat{\Sigma}} = 0$ by assumption, we obtain from item $(4)$ of \cref{Lifting_op_single} that there must hold $\restrict{\hat{L}_{\hat{\Gamma}} f}{\hat{\Sigma}^+} = 0$. We mention that the first case occurs $d-k-1$ times, since $\hat{\Sigma}^+$ occupies $k+2$ nodes, so that the remaining node in $\Nodes(\hat{T}) \backslash \hat{\Gamma}$ must be one of the $(d+1)-(k+2) = d-k-1$ unoccupied nodes of $\hat{T}$. Altogether, it follows that
\begin{equation} \label{Lifting_op_Eq_2}
\restrict{\hat{M} f}{\hat{\Sigma}^+} = \sum_{\hat{\Gamma}: \, \hat{\Sigma}^+ \subseteq \hat{\Gamma}} \restrict{\hat{L}_{\hat{\Gamma}} f}{\hat{\Sigma}^+} + \sum_{\hat{\Gamma}: \, \hat{\Sigma}^+ \not\subseteq \hat{\Gamma}} \restrict{\hat{L}_{\hat{\Gamma}} f}{\hat{\Sigma}^+} = (d-k-1) \restrictN{f}{\hat{\Sigma}^+}.
\end{equation}

Finally, for all $f \in \Ck{0}{\boundaryN{\hat{T}}}$, we may use item $(5)$ of \cref{Lifting_op_single} to derive a stability bound for the operator $\hat{M}$:
\begin{equation} \label{Lifting_op_Eq_3}
\norm[\Lp{2}{\hat{T}}]{\hat{M} f} \leq \sum_{\hat{\Gamma}} \norm[\Lp{2}{\hat{T}}]{\hat{L}_{\hat{\Gamma}} f} \cleq p^{-1/2} \norm[\Lp{2}{\boundaryN{\hat{T}}}]{f}.
\end{equation}

Our presentation of the auxiliary operator $\hat{M}$ is now finished and we proceed to construct $\hat{L}$ from $\hat{M}$. Let $f \in \Ck{0}{\boundaryN{\hat{T}}}$ be given. We use the restriction operator $\hat{R} := \restrict{\cdot}{\boundaryN{\hat{T}}}$ and the coefficients $c_k := (d-k)^{-1}$, $k \in \set{0,\dots,d-1}$, to define an auxiliary function:
\begin{equation*}
\tilde{f} := (\identity-c_{d-1}\hat{R}\hat{M}) \cdots (\identity-c_1\hat{R}\hat{M}) (\identity-c_0\hat{R}\hat{M}) f \in \Ck{0}{\boundaryN{\hat{T}}}.
\end{equation*}

From \cref{Lifting_op_Eq_1} we know that the function $(\identity-c_0\hat{R}\hat{M}) f \in \Ck{0}{\boundaryN{\hat{T}}}$ vanishes on all $0$-simplices of $\hat{T}$. Then, using \cref{Lifting_op_Eq_2}, we may conclude that $(\identity-c_1\hat{R}\hat{M})(\identity-c_0\hat{R}\hat{M}) f \in \Ck{0}{\boundaryN{\hat{T}}}$ vanishes on all $1$-simplices of $\hat{T}$. Proceeding forwards with \cref{Lifting_op_Eq_2}, we find that $\tilde{f}$ must vanishes on all $(d-1)$-simplices of $\hat{T}$. However, since the $(d-1)$-simplices make up all of $\boundaryN{\hat{T}}$, we have $\tilde{f} = 0 \in \Ck{0}{\boundaryN{\hat{T}}}$. Expanding $\tilde{f}$, we find that, for certain coefficients $\tilde{c}_k \in \R$,
\begin{equation*}
0 = \tilde{f} = f - \sum_{k=1}^{d} \tilde{c}_k (\hat{R}\hat{M})^k f = f - \hat{R}\hat{M} \sum_{k=0}^{d-1} \tilde{c}_{k+1} (\hat{R}\hat{M})^k f.
\end{equation*}

Now, define
\begin{equation*}
\hat{L} f := \hat{M} \sum_{k=0}^{d-1} \tilde{c}_{k+1} (\hat{R}\hat{M})^k f \in \Ck{0}{\hat{T}}.
\end{equation*}

Clearly, $\restrict{\hat{L} f}{\boundaryN{\hat{T}}} = \hat{R}\hat{L} f = f$. Furthermore, since $\hat{R}$ and $\hat{M}$ preserve polynomials, so does $\hat{L}$. Finally, let us derive a bound for $\hat{L} f$ in the case of a polynomial input, $f \in \Pp{p}{\boundaryN{\hat{T}}}$. Since the powers $(\hat{R}\hat{M})^k f$ are polynomials as well, it pays off to have a look at $\hat{R}\hat{M} g$, where $g \in \Pp{p}{\boundaryN{\hat{T}}}$. Using a multiplicative trace inequality, \cite{BS02}, and an inverse inequality (e.g., \cite{Ditzian_Inverse_inequality}), we compute
\begin{equation*}
\norm[\Lp{2}{\boundaryN{\hat{T}}}]{\hat{R}\hat{M} g} = \norm[\Lp{2}{\boundaryN{\hat{T}}}]{\hat{M} g} \cleq \norm[\Lp{2}{\hat{T}}]{\hat{M} g}^{1/2} \norm[\Hk{1}{\hat{T}}]{\hat{M} g}^{1/2} \cleq p \norm[\Lp{2}{\hat{T}}]{\hat{M} g} \stackrel{\cref{Lifting_op_Eq_3}}{\cleq} p^{1/2} \norm[\Lp{2}{\boundaryN{\hat{T}}}]{g}.
\end{equation*}

We conclude the proof with the bound for $\hat{L} f$:
\begin{equation*}
\norm[\Lp{2}{\hat{T}}]{\hat{L} f} \stackrel{\cref{Lifting_op_Eq_3}}{\cleq} p^{-1/2} \sum_{k=0}^{d-1} \norm[\Lp{2}{\boundaryN{\hat{T}}}]{(\hat{R}\hat{M})^k f} \cleq p^{-1/2} \sum_{k=0}^{d-1} p^{k/2} \norm[\Lp{2}{\boundaryN{\hat{T}}}]{f} \cleq p^{(d-2)/2} \norm[\Lp{2}{\boundaryN{\hat{T}}}]{f}.
\end{equation*}

\end{proof}

The next lemma generalizes the results from \cite{Melenk_Rojik} to arbitrary space dimensions $d \geq 1$. The approach taken here is slightly different from \cite{Melenk_Rojik}, since we define the operator $\hat{J}^p$ by induction on $d$. Furthermore, as was pointed out at the beginning of this section, it suffices to consider polynomial inputs $f \in \Pp{p+2}{\hat{T}}$.

\begin{lemma} \label{Rojik_operator}
There exists a linear operator $\fDef{\hat{J}^p}{\Pp{p+2}{\hat{T}}}{\Pp{p}{\hat{T}}}$ with the following properties:
\begin{enumerate}
\item For all $k \in \set{0,\dots,d}$, all $k$-simplices $\hat{\Sigma} \subseteq \hat{T}$ and all $f \in \Pp{p+2}{\hat{T}}$, the quantity $\restrict{\hat{J}^p f}{\hat{\Sigma}}$ is uniquely determined by $\restrictN{f}{\hat{\Sigma}}$.

\item $\hat{J}^p$ is a projection, i.e., $\hat{J}^p f = f$ for all $f \in \Pp{p}{\hat{T}}$.

\item For all $f \in \Pp{p+2}{\hat{T}}$, there hold the following stability and error bounds:
\begin{eqnarray*}
\norm[\Lp{2}{\hat{T}}]{\hat{J}^p f} &\cleq& p^{d(d+1)/4} \norm[\Lp{2}{\hat{T}}]{f}, \\
\norm[\Lp{2}{\hat{T}}]{f-\hat{J}^p f} &\cleq& p^{d(d+1)/4} \inf_{g \in \Pp{p}{\hat{T}}} \norm[\Lp{2}{\hat{T}}]{f-g}.
\end{eqnarray*}

\end{enumerate}
\end{lemma}

\begin{proof}
We construct the operator $\hat{J}^p$ via induction on the space dimension $d \geq 1$ and write $\hat{J}^p_1$, $\hat{J}^p_2$, \dots, $\hat{J}^p_d$ for the corresponding operators. As part of the induction argument, we prove item $(1)$, item $(2)$ and the stability bound from item $(3)$. Finally, the error bound is \emph{not} part of the induction, since it follows readily from the projection property and the stability bound.

The case $d=1$: Denote by $\fDef{\hat{L}}{\Ck{0}{\boundaryN{\hat{T}}}}{\Ck{0}{\hat{T}}}$ the polynomial preserving lifting operator from \cref{Lifting_op}. Note that, since $d=1$, we have $\hat{L} f \in \Pp{p}{\hat{T}}$, for \emph{all} $f \in \Ck{0}{\boundaryN{\hat{T}}}$. Let $\PpO{p}{\hat{T}} := \Set{f \in \Pp{p}{\hat{T}}}{\restrictN{f}{\boundaryN{\hat{T}}} = 0}$ and denote by $\fDef{\hat{P}}{\Lp{2}{\hat{T}}}{\PpO{p}{\hat{T}}}$ the orthogonal projection. We define
\begin{equation*}
\forall f \in \Pp{p+2}{\hat{T}}: \quad \quad \hat{J}^p_1 f := \hat{L}(\restrictN{f}{\boundaryN{\hat{T}}}) + \hat{P}(f - \hat{L}(\restrictN{f}{\boundaryN{\hat{T}}})) \in \Pp{p}{\hat{T}}.
\end{equation*}

The identity $\restrict{\hat{J}^p_1 f}{\boundaryN{\hat{T}}} = \restrictN{f}{\boundaryN{\hat{T}}}$, for all $f \in \Pp{p+2}{\hat{T}}$, proves item $(1)$. Since $\hat{P}$ is a projection, so is $\hat{J}^p_1$. Using a multiplicative trace inequality and an inverse inequality, we obtain, for all $f \in \Pp{p+2}{\hat{T}}$, the stability bound
\begin{eqnarray*}
\norm[\Lp{2}{\hat{T}}]{\hat{J}^p_1 f} &\leq& \norm[\Lp{2}{\hat{T}}]{\hat{L}(\restrictN{f}{\boundaryN{\hat{T}}})} + \norm[\Lp{2}{\hat{T}}]{\hat{P}(f - \hat{L}(\restrictN{f}{\boundaryN{\hat{T}}}))} \cleq \norm[\Lp{2}{\hat{T}}]{\hat{L}(\restrictN{f}{\boundaryN{\hat{T}}})} + \norm[\Lp{2}{\hat{T}}]{f} \\
&\stackrel{\cref{Lifting_op}}{\cleq}& p^{-1/2} \norm[\lp{2}{\boundaryN{\hat{T}}}]{f} + \norm[\Lp{2}{\hat{T}}]{f} \cleq p^{-1/2} \norm[\Lp{2}{\hat{T}}]{f}^{1/2} \norm[\Hk{1}{\hat{T}}]{f}^{1/2} + \norm[\Lp{2}{\hat{T}}]{f} \cleq p^{1/2} \norm[\Lp{2}{\hat{T}}]{f}.
\end{eqnarray*}

The step $d-1 \mapsto d$: Assume that an operator $\fDef{\hat{J}^p_{d-1}}{\Pp{p+2}{\hat{T}^{d-1}}}{\Pp{p}{\hat{T}^{d-1}}}$ satisfying items $(1)$, $(2)$ and the stability bound from $(3)$ is well-defined. Furthermore, let us denote the $(d-1)$-subsimplices of $\hat{T}$ by $\hat{\Gamma}_0,\dots,\hat{\Gamma}_d$ and fix affine parametrizations $\fDef{\gamma_i}{\hat{T}^{d-1}}{\hat{\Gamma}_i}$. In order to construct the operator $\fDef{\hat{J}^p_d}{\Pp{p+2}{\hat{T}^d}}{\Pp{p}{\hat{T}^d}}$ from $\hat{J}^p_{d-1}$, we proceed roughly as follows: Given $f \in \Pp{p+2}{\hat{T}}$, we can use $\hat{J}^p_{d-1}$ to find a polynomial $g \in \Pp{p}{\boundaryN{\hat{T}}}$ with $g \approx \restrictN{f}{\boundaryN{\hat{T}}}$. Then, using the polynomial preserving lifting operator $\fDef{\hat{L}}{\Ck{0}{\boundaryN{\hat{T}}}}{\Ck{0}{\hat{T}}}$ from \cref{Lifting_op}, we introduce the quantity $G := \hat{L}g \in \Pp{p}{\hat{T}}$. Clearly, $\restrictN{G}{\boundaryN{\hat{T}}} \approx \restrictN{f}{\boundaryN{\hat{T}}}$, but \emph{not} necessarily $G \approx f$ in all of $\hat{T}$. However, if $\fDef{\hat{P}}{\Lp{2}{\hat{T}}}{\PpO{p}{\hat{T}}}$ denotes the orthogonal projection onto the space of homogeneous polynomials $\PpO{p}{\hat{T}}$, then indeed $\hat{J}^p_d f := G+\hat{P}(f-G) \approx G+(f-G) = f$ on $\hat{T}$. 

In order to work out the details, let $f \in \Pp{p+2}{\hat{T}}$ be given. Then, for each $i \in \set{0,\dots,d}$, $f \circ \gamma_i \in \Pp{p+2}{\hat{T}^{d-1}}$, so that the polynomial $\hat{J}^p_{d-1}(f \circ \gamma_i) \in \Pp{p}{\hat{T}^{d-1}}$ is well-defined by the induction hypothesis. We define boundary data $\fDef{g}{\boundaryN{\hat{T}}}{\R}$ in a piecewise manner:
\begin{equation*}
\forall i \in \set{0,\dots,d}: \quad \quad \restrictN{g}{\hat{\Gamma}_i} := \hat{J}^p_{d-1}(f \circ \gamma_i) \circ \gamma_i^{-1} \in \Pp{p}{\hat{\Gamma}_i}.
\end{equation*}
(Note that $\gamma_i$ is injective and thus invertible on its range, $\hat{\Gamma}_i$.)

We argue that $g \in \Pp{p}{\boundaryN{\hat{T}}}$: Consider the boundary $\hat{\Sigma}_{ij} := \hat{\Gamma}_i \cap \hat{\Gamma}_j \subseteq \hat{T}$ between any two $(d-1)$-simplices $\hat{\Gamma}_i,\hat{\Gamma}_j$ and note that $\hat{\Sigma}_{ij}$ is a $(d-2)$-simplex. Then, the pre-images $\gamma_i^{-1}(\hat{\Sigma}_{ij}), \gamma_j^{-1}(\hat{\Sigma}_{ij}) \subseteq \hat{T}^{d-1}$ are $(d-2)$-simplices as well. Using item $(1)$ of the induction hypothesis, we know that $\restrict{\hat{J}^p_{d-1}(f \circ \gamma_i)}{\gamma_i^{-1}(\hat{\Sigma}_{ij})}$ is uniquely determined by $\restrict{f \circ \gamma_i}{\gamma_i^{-1}(\hat{\Sigma}_{ij})}$ and that $\restrict{\hat{J}^p_{d-1}(f \circ \gamma_j)}{\gamma_j^{-1}(\hat{\Sigma}_{ij})}$ is uniquely determined by $\restrict{f \circ \gamma_j}{\gamma_j^{-1}(\hat{\Sigma}_{ij})}$. However, since $\restrict{f \circ \gamma_i}{\gamma_i^{-1}(\hat{\Sigma}_{ij})} = \restrictN{f}{\hat{\Sigma}_{ij}} = \restrict{f \circ \gamma_j}{\gamma_j^{-1}(\hat{\Sigma}_{ij})}$, there must hold $\restrict{\hat{J}^p_{d-1}(f \circ \gamma_i)}{\gamma_i^{-1}(\hat{\Sigma}_{ij})} = \restrict{\hat{J}^p_{d-1}(f \circ \gamma_j)}{\gamma_j^{-1}(\hat{\Sigma}_{ij})}$. It follows that $\restrictN{\restrictN{g}{\hat{\Gamma}_i}}{\hat{\Sigma}_{ij}} = \restrictN{\restrictN{g}{\hat{\Gamma}_j}}{\hat{\Sigma}_{ij}}$, i.e., that $g \in \Ck{0}{\boundaryN{\hat{T}}}$. According to \cref{Space_Pp_boundary}, it follows that $g \in \Pp{p}{\boundaryN{\hat{T}}}$.

Furthermore, to get a stability estimate for $g$, we can use item $(3)$ of the induction hypothesis:
\begin{eqnarray*}
\norm[\Lp{2}{\boundaryN{\hat{T}}}]{g} &\cleq& \sum_{i=0}^{d} \norm[\Lp{2}{\hat{\Gamma}_i}]{\hat{J}^p_{d-1}(f \circ \gamma_i) \circ \gamma_i^{-1}} \cleq \sum_{i=0}^{d} \norm[\Lp{2}{\hat{T}^{d-1}}]{\hat{J}^p_{d-1}(f \circ \gamma_i)} \\
&\stackrel{(3)}{\cleq}& p^{(d-1)d/4} \sum_{i=0}^{d} \norm[\Lp{2}{\hat{T}^{d-1}}]{f \circ \gamma_i} \cleq p^{(d-1)d/4} \norm[\Lp{2}{\boundaryN{\hat{T}}}]{f}.
\end{eqnarray*}

We proceed as stated above and lift $g$ from $\boundaryN{\hat{T}}$ into $\hat{T}$. From \cref{Lifting_op}, we know that the function $G := \hat{L} g$ satisfies
\begin{equation*}
\restrictN{G}{\boundaryN{\hat{T}}} = g, \quad \quad \quad G \in \Pp{p}{\hat{T}}, \quad \quad \quad \norm[\Lp{2}{\hat{T}}]{G} \cleq p^{(d-2)/2} \norm[\Lp{2}{\boundaryN{\hat{T}}}]{g}.
\end{equation*}

Now, recalling that $\fDef{\hat{P}}{\Lp{2}{\hat{T}}}{\PpO{p}{\hat{T}}}$ denotes the orthogonal projection, consider the function
\begin{equation*}
\hat{J}^p_d f := G + \hat{P}(f-G) \in \Pp{p}{\hat{T}}.
\end{equation*}
Clearly, the mapping $f \mapsto \hat{J}^p_d f$ defines a linear operator $\fDef{\hat{J}^p_d}{\Pp{p+2}{\hat{T}}}{\Pp{p}{\hat{T}}}$. To prove item $(1)$, let $k \in \set{0,\dots,d}$ and consider a $k$-simplex $\hat{\Sigma} \subseteq \hat{T}$. If $k=d$, the statement becomes trivial. If $k \leq d-1$, then there exists a $(d-1)$-simplex $\hat{\Gamma}_i \subseteq \hat{T}$ such that $\hat{\Sigma} \subseteq \hat{\Gamma}_i \subseteq \boundaryN{\hat{T}}$. Since $\hat{P}(f-G)$ vanishes on $\boundaryN{\hat{T}}$, we find that
\begin{equation*}
\restrict{\hat{J}^p_d f}{\hat{\Sigma}} = \restrictN{G}{\hat{\Sigma}} = \restrictN{g}{\hat{\Sigma}} = \hat{J}^p_{d-1}(f \circ \gamma_i) \circ (\restrictN{\gamma_i^{-1}}{\hat{\Sigma}}).
\end{equation*}
Item $(1)$ of the induction hypothesis tells us that this function is uniquely determined by $\restrict{f \circ \gamma_i}{\gamma_i^{-1}(\hat{\Sigma})}$, i.e., by $\restrictN{f}{\hat{\Sigma}}$.

As for the projection property of $\hat{J}^p_d$, consider an input $f \in \Pp{p}{\hat{T}}$. Then, $g = \restrictN{f}{\boundaryN{\hat{T}}} \in \Pp{p}{\boundaryN{\hat{T}}}$, since $\hat{J}^p_{d-1}$ is a projection by the induction hypothesis. It follows that $f-G \in \PpO{p}{\hat{T}}$ so that $\hat{J}^p_d f = G + \hat{P}(f-G) = G + (f-G) = f$.

Finally, for all $f \in \Pp{p+2}{\hat{T}}$, a multiplicative trace inequality and an inverse inequality give us the desired stability estimate:
\begin{equation*}
\begin{array}{rclcl}
\norm[\Lp{2}{\hat{T}}]{\hat{J}^p_d f} &\leq& \norm[\Lp{2}{\hat{T}}]{G} + \norm[\Lp{2}{\hat{T}}]{\hat{P}(f-G)} &\leq& \norm[\Lp{2}{\hat{T}}]{G} + \norm[\Lp{2}{\hat{T}}]{f-G} \\
&\cleq& \norm[\Lp{2}{\hat{T}}]{f} + \norm[\Lp{2}{\hat{T}}]{G} &\cleq& \norm[\Lp{2}{\hat{T}}]{f} + p^{(d-2)/2} \norm[\Lp{2}{\boundaryN{\hat{T}}}]{g} \\
&\cleq& \norm[\Lp{2}{\hat{T}}]{f} + p^{(d-2)/2 + (d-1)d/4} \norm[\Lp{2}{\boundaryN{\hat{T}}}]{f} &\cleq& \norm[\Lp{2}{\hat{T}}]{f} + p^{(d-2)/2 + (d-1)d/4} \norm[\Lp{2}{\hat{T}}]{f}^{1/2} \norm[\Hk{1}{\hat{T}}]{f}^{1/2} \\
&\cleq& p^{(d-2)/2 + (d-1)d/4 + 1} \norm[\Lp{2}{\hat{T}}]{f} &=& p^{d(d+1)/4} \norm[\Lp{2}{\hat{T}}]{f}.
\end{array}
\end{equation*}

This finishes the proof.

\end{proof}

We close this section with the delayed proof of \cref{Melenk_Rojik}.
\begin{proof}[Proof of \cref{Melenk_Rojik}]
Denote by $\fDef{\hat{J}^p}{\Pp{p+2}{\hat{T}}}{\Pp{p}{\hat{T}}}$ the operator from \cref{Rojik_operator}. We define the asserted operator $\fDef{J_{\Elements}^p}{\Skp{0}{p+2}{\Elements}}{\Skp{0}{p}{\Elements}}$ in an elementwise fashion: For every $v \in \Skp{0}{p+2}{\Elements}$ and every element $T \in \Elements$, we set
\begin{equation*}
\restrict{J_{\Elements}^p v}{T} := \hat{J}^p(v \circ F_T) \circ F_T^{-1}.
\end{equation*}
(Recall from \cref{Mesh} that $\fDef{F_T}{\hat{T}}{T}$ is the affine transformation between $\hat{T}$ and $T$.)

The preservation of continuity and boundary values follows from item $(1)$ in \cref{Rojik_operator}. The preservation of supports is obvious from the elementwise definition. Finally, to see the error bound, let $\kappa \in \Pp{1}{\hat{T}}$ and $u \in \Pp{p}{\hat{T}}$. Then, using an inverse inequality once again, we obtain
\begin{equation*}
\begin{array}{rclcl}
\norm[\Hk{1}{\hat{T}}]{(\identity-\hat{J}^p)(\kappa^2 u)} &\cleq& p^2 \norm[\Lp{2}{\hat{T}}]{(\identity-\hat{J}^p)(\kappa^2 u)} &\stackrel{\cref{Rojik_operator}}{\cleq}& p^{d(d+1)/4+2} \inf_{g \in \Pp{p}{\hat{T}}} \norm[\Lp{2}{\hat{T}}]{\kappa^2 u - g} \\
&\leq& p^{d(d+1)/4+2} \norm[\Lp{2}{\hat{T}}]{\kappa^2 u - \kappa(0)^2 u} &\stackrel{\mathrm{Taylor}}{\cleq}& p^{d(d+1)/4+2} \seminorm[\Wkp{1}{\infty}{\hat{T}}]{\kappa^2} \norm[\Lp{2}{\hat{T}}]{u}.
\end{array}
\end{equation*}

Item $(3)$ of \cref{Melenk_Rojik} then follows with the standard scaling relation $\h{T}^l \seminorm[\Wkp{l}{q}{T}]{v} \ceq \h{T}^{d/q} \seminorm[\Wkp{l}{q}{\hat{T}}]{v \circ F_T}$. In fact, if $\hat{\kappa} := \kappa \circ F_T$ and $\hat{u} := u \circ F_T$ denote the pull-backs of $\kappa$ and $u$, then
\begin{eqnarray*}
\sum_{l=0}^{1} \h{T}^l \seminorm[\Hk{l}{T}]{(\identity-J_{\Elements}^p)(\kappa^2 u)} &\cleq& \h{T}^{d/2} \norm[\Hk{1}{\hat{T}}]{(\identity-\hat{J}^p)(\hat{\kappa}^2 \hat{u})} \cleq p^{d(d+1)/4+2} \h{T}^{d/2} \seminorm[\Wkp{1}{\infty}{\hat{T}}]{\hat{\kappa}^2} \norm[\Lp{2}{\hat{T}}]{\hat{u}} \\
&\ceq& p^{d(d+1)/4+2} \h{T} \seminorm[\Wkp{1}{\infty}{T}]{\kappa^2} \norm[\Lp{2}{T}]{u}.
\end{eqnarray*}
This concludes the proof of \cref{Melenk_Rojik}.

\end{proof}

\section{Numerical results}

In this final section, we illustrate the validity of \cref{Main_result} with two numerical examples in $d=2$ space dimensions.
The domain $\Omega := (0,1) \times (0,1) \subseteq \R^2$ is triangulated with a mesh $\Elements$ with exponential grading towards the left edge $\Gamma := \set{0} \times [0,1]$ (cf. \cref{SSec_The_main_result}). Each element $T \in \Elements$ satisfies $\h{T} \ceq \dist[2]{x_T}{\Gamma}^{1-1/\alpha} H$, where $\alpha = \infty$ and $H = 0.25$.

We start with the special case $p=1$.   The system matrix $\mvemph{A} \in \R^{N \times N}$ is assembled and explicitly inverted using MATLAB's built-in inversion routine $\texttt{inv(\dots)}$. Then, for each rank bound $r \in \set{1,\dots,15}$, an approximation $\mvemph{B}_r \in \HMatrices{\BPart}{r}$ to $\mvemph{A}^{-1}$ is computed via blockwise truncated singular values decompositions. As was discussed in more detail in \cite[Section 4]{Angleitner_H_matrices_FEM}, this procedure gives rise to the \emph{computable} error bound
\begin{equation*}
\norm[2]{\mvemph{A}^{-1}-\mvemph{B}_r} \cleq \depth{\Tree{N \times N}} \cdot \max_{(I,J) \in \BPart} \sigma_{r+1}(\restrictN{\mvemph{A}^{-1}}{I \times J}).
\end{equation*}

\begin{figure}[H]
\begin{center}
\includegraphics[width=\textwidth, trim=0cm 0cm 0cm 0cm]{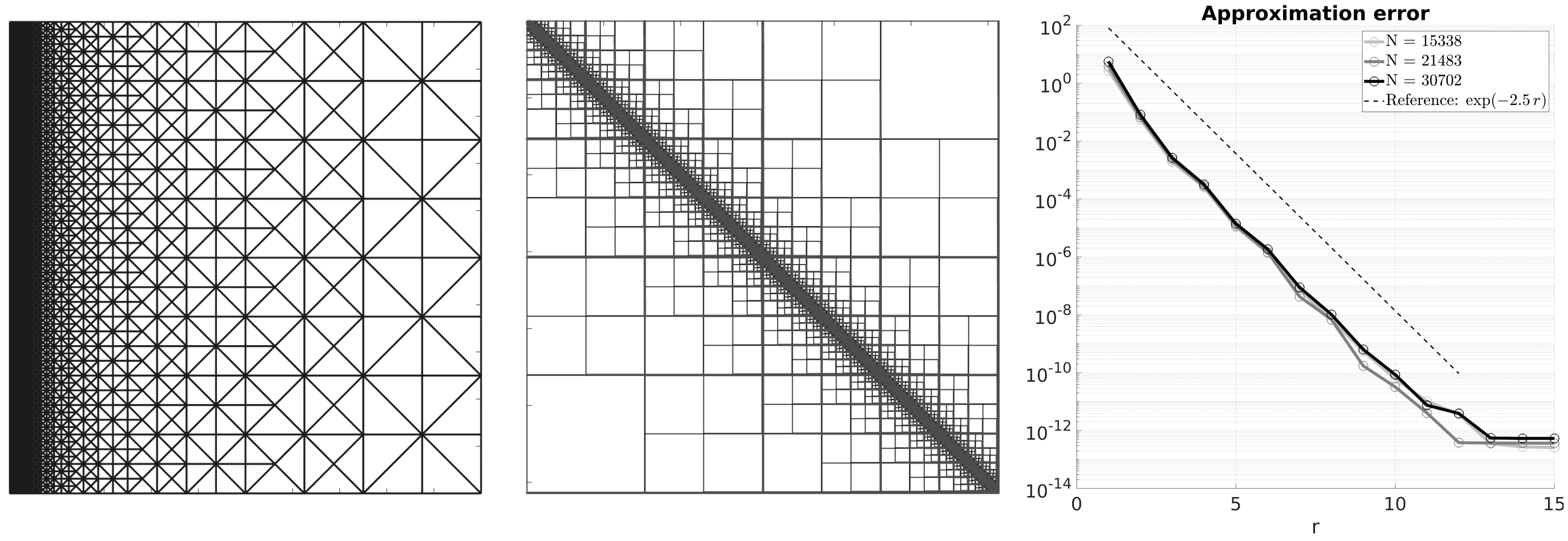}
\caption{Left: The mesh $\Elements$. Center: The block partition $\BPart$. Right: Empirical approximation errors.}
\label{Num_example_1}
\end{center}
\end{figure}

The right-hand image in \cref{Num_example_1} depicts a comparison between three different problem sizes of roughly $N \approx 15.000$, $N \approx 21.500$ and $N \approx 31.000$ degrees of freedom. The error appears to decline at a rate of $\exp(-2.5 r)$, which is even better than our theoretical prediction $\exp(-\CExp r^{1/3})$ from \cref{Main_result}.

\begin{figure}[H]
\begin{center}
\includegraphics[width=0.5\textwidth, trim=0cm 0cm 0cm 0cm]{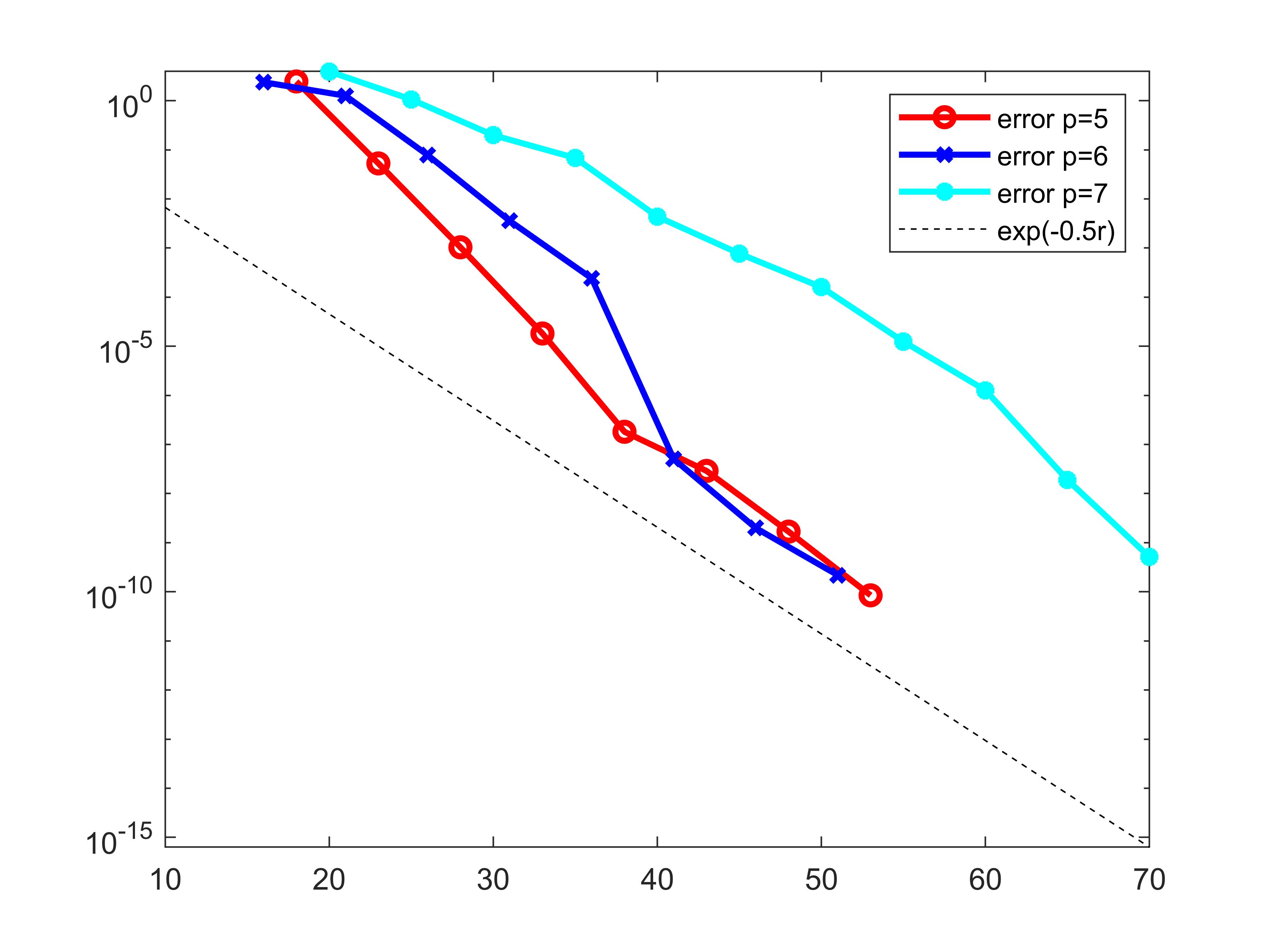} 
\caption{Exponential convergence of $\mathcal{H}$-matrix approximations for $p=5,6,7$.}
\label{Figure_bcFEM}
\end{center}
\end{figure}

With the previously defined exponentially graded mesh towards the edge $x=0$, we also compute an example with higher polynomial degrees on each element. 
We employ a combination of the finite element code {\tt NGSolve} (which is capable of higher order polynomials), \cite{NGSolve}, and the C++ $\mathcal{H}$-matrix library, \cite{H2Lib}. Hereby, both codes are coupled using a code also employed in \cite{Rieder}. 
We use polynomial degrees $p=5$ (which leads to a problem size of $N=5791$), $p=6$ (which leads to a problem size of $N=17053$) and $p=7$ (which leads to a problem size of $N=46915$). The $\mathcal{H}$-matrix approximations are computed using $\mathcal{H}$-Cholesky decompositions and then inverting the Cholesky factors. In order to avoid computing the full inverse matrix, we compute the error measure $\norm[2]{\mvemph{I} - (\mvemph{C}_{\mathcal{H}}\mvemph{C}_{\mathcal{H}}^T)^{-1} \mvemph{A}}$, which is an upper bound for the relative error.

Figure~\ref{Figure_bcFEM} shows exponential convergence of the error measure as predicted by our main result.

\bibliographystyle{amsalpha}
\bibliography{literature}

\newcommand{\etalchar}[1]{$^{#1}$}
\providecommand{\bysame}{\leavevmode\hbox to3em{\hrulefill}\thinspace}
\providecommand{\MR}{\relax\ifhmode\unskip\space\fi MR }
\providecommand{\MRhref}[2]{%
  \href{http://www.ams.org/mathscinet-getitem?mr=#1}{#2}
}
\providecommand{\href}[2]{#2}
\begin{thebibliography}{EMM{\etalchar{+}}21}

\bibitem[AFM21a]{Angleitner_H_matrices_FEM}
N.~Angleitner, M.~Faustmann, and J.M. Melenk, \emph{Approximating inverse {FEM}
  matrices on non-uniform meshes with $\mathcal{H}$-matrices}, Calcolo
  \textbf{58} (2021), no.~3, Paper No. 31, 36.

\bibitem[AFM21b]{Angleitner_H_matrices_RBF}
\bysame, \emph{$\mathcal{H}$-inverses for rbf interpolation}, arXiv e-prints
  \textbf{no. arXiv:2109.05763} (2021).

\bibitem[BCMP91]{babuska-craig-mandel-pitkaranta91}
I.~Babu{\v s}ka, A.~Craig, J.~Mandel, and J.~Pitk\"aranta, \emph{Efficient
  preconditioning for the $p$ version finite element method in two dimensions},
  SIAM J.~Numer.~Anal. \textbf{28} (1991), no.~3, 624--661.

\bibitem[BDM92]{bernardi-dauge-maday92}
C.~Bernardi, M.~Dauge, and Y.~Maday, \emph{Trace liftings which preserve
  polynomials}, C.R. Acad. Sci. Paris, S{\'e}rie {I} \textbf{315} (1992),
  333--338.

\bibitem[BDM07]{bernardi-dauge-maday07}
\bysame, \emph{Polynomials in the {S}obolev world (version 2)}, Tech.
  Report~14, IRMAR, 2007, https://hal.archives-ouvertes.fr/hal-00153795.

\bibitem[Beb08]{bebendorfbook}
M.~Bebendorf, \emph{{H}ierarchical {M}atrices}, Lecture Notes in Computational
  Science and Engineering, vol.~63, Springer, Berlin, 2008.

\bibitem[BH03]{BebendorfHackbusch03}
M.~Bebendorf and W.~Hackbusch, \emph{Existence of {$\mathcal{H}$}-matrix
  approximants to the inverse {FE}-matrix of elliptic operators with
  {$L^{\infty}$}-coefficients}, Numer. Math. \textbf{95} (2003), no.~1, 1--28.

\bibitem[BM97]{bernardi-maday97}
C.~Bernardi and Y.~Maday, \emph{Spectral methods}, Handbook of Numerical
  Analysis, Vol. 5 (P.G. Ciarlet and J.L. Lions, eds.), North Holland,
  Amsterdam, 1997.

\bibitem[BO09]{bebendorf2009parallel}
M.~Bebendorf and J.~Ostrowski, \emph{Parallel hierarchical matrix
  preconditioners for the curl-curl operator}, J. Comput. Math. (2009),
  624--641.

\bibitem[B{\"o}r10]{boermbook}
S.~B{\"o}rm, \emph{Efficient numerical methods for non-local operators}, EMS
  Tracts in Mathematics, vol.~14, European Mathematical Society (EMS),
  Z\"urich, 2010.

\bibitem[B{\"{o}}r21]{H2Lib}
S.~B{\"{o}}rm, \emph{{$\mathcal{H}$}2{LIB} software library}, University of
  Kiel, http://www.h2lib.org (2021).

\bibitem[BS02]{BS02}
S.C. Brenner and L.R. Scott, \emph{The mathematical theory of finite element
  methods}, Texts in Applied Mathematics, vol.~15, Springer-Verlag, New York,
  2002.

\bibitem[BSK81]{Katz_Polynomial_lifting}
I.~Babu\v{s}ka, B.~A. Szabo, and I.~N. Katz, \emph{The {$p$}-version of the
  finite element method}, SIAM J. Numer. Anal. \textbf{18} (1981), no.~3,
  515--545.

\bibitem[CD05]{demkowicz-cao05}
W.~Cao and L.~Demkowicz, \emph{Optimal error estimate of a projection based
  interpolation for the {$p$}-version approximation in three dimensions},
  Comput. Math. Appl. \textbf{50} (2005), no.~3-4, 359--366.

\bibitem[Cia78]{Ciarlet_FEM_Meshes}
P.G. Ciarlet, \emph{The finite element method for elliptic problems},
  North-Holland Publishing Co., Amsterdam-New York-Oxford, 1978, Studies in
  Mathematics and its Applications, Vol. 4.

\bibitem[Cl{\'e}75]{clement75}
Ph. Cl{\'e}ment, \emph{Approximation by finite element functions using local
  regularization}, Rev. Fran\c{c}aise Automat. Informat. Recherche
  Op\'{e}rationnelle S\'{e}r. \textbf{9} (1975), no.~{\rm R}-2, 77--84.

\bibitem[DB03]{demkowicz-babuska03}
L.~Demkowicz and I.~Babu\v{s}ka, \emph{{$p$} interpolation error estimates for
  edge finite elements of variable order in two dimensions}, SIAM J. Numer.
  Anal. \textbf{41} (2003), no.~4, 1195--1208.

\bibitem[DB05]{demkowicz-buffa05}
L.~Demkowicz and A.~Buffa, \emph{{$H^1$}, {$H({\rm curl})$} and {$H({\rm
  div})$}-conforming projection-based interpolation in three dimensions.
  {Q}uasi-optimal {$p$}-interpolation estimates}, Comput. Methods Appl. Mech.
  Engrg. \textbf{194} (2005), no.~2-5, 267--296.

\bibitem[Dem08]{demkowicz08}
L.~Demkowicz, \emph{Polynomial exact sequences and projection-based
  interpolation with applications to {M}axwell's equations}, Mixed Finite
  Elements, Compatibility Conditions, and Applications (D.~Boffi, F.~Brezzi,
  L.~Demkowicz, L.F. Dur{\'a}n, R.~Falk, and M.~Fortin, eds.), Lectures Notes
  in Mathematics, vol. 1939, Springer Verlag, 2008.

\bibitem[DHS17]{doelz-harbrecht-schwab17}
J.~D\"{o}lz, H.~Harbrecht, and Ch. Schwab, \emph{Covariance regularity and
  {$\mathcal{H}$}-matrix approximation for rough random fields}, Numer. Math.
  \textbf{135} (2017), no.~4, 1045--1071.

\bibitem[Dit92]{Ditzian_Inverse_inequality}
Z.~Ditzian, \emph{Multivariate {B}ernstein and {M}arkov inequalities}, J.
  Approx. Theory \textbf{70} (1992), no.~3, 273--283.

\bibitem[EG00]{Edelsbrunner}
H.~Edelsbrunner and D.R. Grayson, \emph{Edgewise subdivision of a simplex},
  vol.~24, 2000, ACM Symposium on Computational Geometry (Miami, FL, 1999),
  pp.~707--719.

\bibitem[EMM{\etalchar{+}}21]{Rieder}
C.~Erath, L.~Mascotto, J.M. Melenk, I.~Perugia, and A.~Rieder, \emph{Mortar
  coupling of $hp$-discontinuous galerkin and boundary element methods for the
  helmholtz equation}, arXiv e-prints \textbf{no. arXiv:2105.06173} (2021).

\bibitem[FMP15]{Faustmann_H_matrices_FEM}
M.~Faustmann, J.M. Melenk, and D.~Praetorius, \emph{H-matrix approximability of
  the inverses of {FEM} matrices}, Numer. Math. \textbf{131} (2015), no.~4,
  615--642.

\bibitem[FMP16]{FMP16}
\bysame, \emph{Existence of {${\mathcal H}$}-matrix approximants to the inverse
  of {BEM} matrices: the simple-layer operator}, Math. Comp. \textbf{85}
  (2016), 119--152.

\bibitem[FMP17]{FMP17}
\bysame, \emph{Existence of {${\mathcal H}$}-matrix approximants to the inverse
  of {BEM} matrices: the hyper-singular integral operator}, IMA J. Numer. Anal.
  \textbf{37} (2017), no.~3, 1211--1244.

\bibitem[FMP20]{FMP21}
M.~Faustmann, J.M. Melenk, and M.~Parvizi, \emph{Caccioppoli-type estimates and
  {$\mathcal{H}$}-matrix approximations to inverses for {FEM}-{ BEM}
  couplings}, arXiv e-prints \textbf{no. arXiv:2008.11498} (2020).

\bibitem[FMP21]{faustmann2021mathcalhmatrix}
M.~Faustmann, J.M. Melenk, and M.~Parvizi, \emph{$\mathcal{H}$-matrix
  approximability of inverses of {FEM} matrices for the time-harmonic {M}axwell
  equations}, arXiv e-prints \textbf{no. arXiv:2103.14981} (2021).

\bibitem[FMPR15]{Melenk_hp_bases}
T.~F\"{u}hrer, J.M. Melenk, D.~Praetorius, and A.~Rieder, \emph{Optimal
  additive {S}chwarz methods for the {$hp$}-{BEM}: the hypersingular integral
  operator in 3{D} on locally refined meshes}, Comput. Math. Appl. \textbf{70}
  (2015), no.~7, 1583--1605.

\bibitem[GH03]{GH03}
L.~Grasedyck and W.~Hackbusch, \emph{Construction and arithmetics of
  {$\mathcal{H}$}-matrices}, Computing \textbf{70} (2003), no.~4, 295--334.

\bibitem[GHLB04]{Hackbusch_Geometric_clustering}
L.~Grasedyck, W.~Hackbusch, and S.~Le~Borne, \emph{Adaptive geometrically
  balanced clustering of {$\mathcal{H}$}-matrices}, Computing \textbf{73}
  (2004), no.~1, 1--23.

\bibitem[Gra01]{grasedyck01}
L.~Grasedyck, \emph{Theorie und {A}nwendungen {H}ierarchischer {M}atrizen},
  Ph.D. thesis, Universit{\"a}t Kiel, 2001.

\bibitem[Hac99]{hackbusch99}
W.~Hackbusch, \emph{A sparse matrix arithmetic based on
  {$\mathcal{H}$}-matrices. {I}ntroduction to {$\mathcal{H}$}-matrices},
  Computing \textbf{62} (1999), no.~2, 89--108.

\bibitem[Hac15]{Hackbusch_Hierarchical_matrices}
\bysame, \emph{Hierarchical matrices: algorithms and analysis}, Springer Series
  in Computational Mathematics, vol.~49, Springer, Heidelberg, 2015.

\bibitem[KM02]{Melenk_Boundary_concentrated_FEM_Direct_solver}
B.N. Khoromskij and J.M. Melenk, \emph{An efficient direct solver for the
  boundary concentrated {FEM} in 2{D}}, Computing \textbf{69} (2002), no.~2,
  91--117.

\bibitem[KM03]{Melenk_Boundary_concentrated_FEM}
\bysame, \emph{Boundary concentrated finite element methods}, SIAM J. Numer.
  Anal. \textbf{41} (2003), no.~1, 1--36.

\bibitem[MnS97]{Munoz_Polynomial_lifting}
R.~Mu\~{n}oz Sola, \emph{Polynomial liftings on a tetrahedron and applications
  to the {$h$}-{$p$} version of the finite element method in three dimensions},
  SIAM J. Numer. Anal. \textbf{34} (1997), no.~1, 282--314.

\bibitem[MR20]{Melenk_Rojik}
J.M. Melenk and C.~Rojik, \emph{On commuting {$p$}-version projection-based
  interpolation on tetrahedra}, Math. Comp. \textbf{89} (2020), no.~321,
  45--87.

\bibitem[NGS]{NGSolve}
\emph{{NGS}olve}, Available at \url{https://ngsolve.org/}.

\end{thebibliography}

\end{document}